\documentclass[10pt]{article}
\usepackage{amsmath,amsthm,amsfonts,amssymb,bm}

\usepackage[light]{antpolt}    \usepackage[T1]{fontenc}
\usepackage{color,graphicx,graphics}
\usepackage{mathrsfs,amssymb,bm,enumerate}
\usepackage[bbgreekl]{mathbbol}
\usepackage{color,soul}
\DeclareSymbolFontAlphabet{\mathbb}{AMSb}
\DeclareSymbolFontAlphabet{\mathbbl}{bbold}
\usepackage{hyperref}

\usepackage{ulem}
\usepackage{xcolor}
\usepackage[pagewise,mathlines,switch]{lineno}

\usepackage{amsbsy}
\DeclareMathAlphabet{\mathpzc}{OT1}{pzc}{m}{it}

\usepackage{mathbbol}
\usepackage{amssymb}             

\DeclareSymbolFontAlphabet{\amsmathbb}{AMSb}%

\usepackage{bbold}

\usepackage{bbold}
\usepackage{bbm}

 \newcommand{\scal}[1]{\left\langle #1 \right\rangle}
 \newcommand{\sett}[1]{\left\{   #1   \right\}}
 \newcommand{\norm}[1]{\left\|   #1   \right\|}
 \newcommand{\abso}[1]{\left|   #1   \right|}
 \newcommand{\paar}[1]{\left( #1 \right)}

\definecolor{blue}{rgb}{0.0, 0.0, 1.0}
\newcommand{\vertiii}[1]{{\left\vert\kern-0.25ex\left\vert\kern-0.25ex\left\vert #1
		\right\vert\kern-0.25ex\right\vert\kern-0.25ex\right\vert}}
\numberwithin{equation}{section}
\newtheorem{theorem}{\quad Theorem}[section]
\newtheorem{lemma}[theorem]{\quad Lemma}

\newtheorem{remark}[theorem]{\quad Remark}
\newtheorem{definition}[theorem]{\quad Definition}

\newcommand{\ngg}{\mathscr{G}}

\newcommand{\nx}{\mathscr{\chi}}

\newcommand{\lt}{{L^2(\mathbb{R})}}
\newcommand{\dd}{\;{\rm d}}

\newcommand{\ff}{\varphi}
\newcommand{\ee}{{\rm e}}
\newcommand{\ii}{{\rm i}}

\newcommand{\e}{\varepsilon}

\newcommand{\lam}{\lambda}

\newcommand{\rr}{\mathbb{R}}

\allowdisplaybreaks

\usepackage{authblk}

\title{Numerical study of a nonlocal nonlinear Schrödinger equation \\ (MMT model) }
\author[1]{Amin Esfahani}
\author[2]{Gulcin M. Muslu}
\affil[1]{Department of Mathematics, Nazarbayev University, Astana 010000, Kazakhstan}
\affil[2]{Istanbul Technical University, Department of Mathematics, Maslak 34469, Istanbul, Turkey}
		
\date{}

\begin{document}

	\maketitle
 \begin{center}
    \textit{Dedicated to the memory of Professor S. Erdogan Suhubi}
\end{center}
 \let\thefootnote\relax\footnotetext{Amin Esfahani (saesfahani@gmail.com, amin.esfahani@nu.edu.kz)}
		\let\thefootnote\relax\footnotetext{Gulcin M. Muslu (gulcin@itu.edu.tr) Corresponding Author}
	\begin{abstract}
		
	In this paper, we study a nonlocal nonlinear Schrödinger equation (MMT model). We investigate the effect of the nonlocal operator appearing in the nonlinearity on the long-term behavior of solutions, and we identify the conditions under which the solutions of the Cauchy problem associated with this equation is bounded globally in time in the energy space. We also explore the dynamical behavior of standing wave solutions. Therefore, we  first numerically generate standing wave solutions of nonlocal nonlinear Schrödinger equation by using the Petviashvili's iteration method and their stability is investigated by the split-step Fourier method.  This equation also has a two-parameter family of standing wave solutions. In a second step, we  meticulously concern with the construction and stability of  a two-parameter family of standing wave solutions numerically. Finally, we investigate the semi-classical limit of the nonlocal nonlinear Schrödinger equation  in both focusing and defocusing cases.

\vspace{4mm}

		
	\end{abstract}

\section{Introduction}
 Weak turbulence theories have been used to estimate wave number spectra of random waves in a variety of complex physical problems, ranging from surface gravity waves in fluids to ion-acoustic waves in plasmas to optical turbulence, among many applications. Due to weak turbulence, the theory focuses on resonant interactions between small-amplitude waves. {On the other hand, the inclusion of fractional derivatives in nonlinear dispersive wave equations allows for the modeling of nonlocal effects and anomalous dispersion, which are not captured by classical integer-order differential equations. In particular, fractional Laplacian-type operators naturally arise in the study of wave propagation in media with power-law dispersion relations, Lévy flights, and turbulence phenomena.}
In this regard, Majda, McLaughlin, and Tabak (MMT) \cite{cmmt,mmt} have introduced the two-parameter family of one-dimensional nonlinear dispersive wave equations
\begin{equation}\label{original}
\ii\psi_t= \tilde{\lambda} D^{\alpha} \psi+ \zeta D^{\gamma /4}
\left( \mid D^{\gamma /4} \psi \mid^2 D^{\gamma /4} \psi \right),
\end{equation}
where $\alpha > 0$ and $ \tilde{\lambda}, \gamma, \zeta$ are real parameters and the operator $D^\gamma=D^\gamma_x$ is defined via the Fourier symbol $|\xi|^\gamma$, i.e.
\[
\paar{D^\gamma f}^\wedge(\xi)=|\xi|^\gamma\hat{f}(\xi),\qquad f\in \mathcal{S},\,\xi\in\rr,
\]
with $\xi\neq0$ when $\gamma<0$.
Here, $\wedge$ is the Fourier transform  and $\mathcal{S}=\mathcal{S}(\rr)$ is the Schwartz class.
This model  assesses the validity of weak turbulence theory for random waves. {The parameter $\alpha$ controls the dispersion relation $\omega(k)=|k|^\alpha$ and $\gamma$ the nonlinearity (see \cite{pns,saut-wang}). In equation \eqref{original}, the fractional dispersion term $D^\alpha \psi$ generalizes the classical Schr\"odinger equation by permitting a wider range of dispersive behaviors, while the nonlinearity involving $D^{\gamma/4}$ allows for a nonlocal interaction structure, which can have significant implications for wave turbulence and the existence of solitary waves.
The MMT model is indeed a
classical example of a generic four-wave Hamiltonian system which contains only the scattering of two waves into two other waves.  The model can be configured to exhibit the linear dispersion relationship of surface gravity waves on a deep, ideal fluid. Interestingly, the MMT model enables direct numerical simulation of the fundamental equations of motion. We refer  also to \cite{pns,saut-wang,ref1,ref2,ref3,ref4,ref5} and references therein  to see the derivation, the importance and other aspects of \eqref{original}.}
The case $\gamma=0$ was studied
by Zakharov and collaborators \cite{Zakharov, zgpd}  since there is no “nonlinear frequency shift” occurring in the Hamiltonian system and
the signs of the energy and mass fluxes for the Kolmogorov spectrum suggest  that the wave turbulence predictions should be most easily realized for this model. Equation \eqref{original} in this case turns into the well-known fractional Schr\"odinger equation (\cite{laskin,laskin1}):
\begin{equation}
\ii\psi_t= \tilde{\lambda} D^{\alpha} \psi+ \zeta
  | \psi|^2     \psi.
\end{equation}
We note that  the sign of $\gamma$ in \eqref{original} is the opposite of \cite{cmmt,mmt}, but agrees
with \cite{germain, Zakharov,zgpd}.
{There is a huge number of results on the various aspects of the above nonlocal wave equation. See for example \cite {elm,germain,pns,ref5,ref3} and references therein.}

{Setting $u= D^{\gamma/4 }\psi$, under certain regularity
assumptions. Then, equation \eqref{original} can be written as
\begin{equation}\label{original2}
\ii  D^{-\gamma/4}u_t= \tilde{\lambda} D^{\alpha-\gamma /4} u+ \zeta
D^{\gamma /4}\left(|u|^2 u \right).
\end{equation}
By applying the operator $D^{\gamma/4}$ to \eqref{original2},
  we obtain the transformed equation
\begin{equation}\label{transformed}
\ii u_t=\tilde{\lambda} D^{\alpha}u+
\zeta D^{\gamma/2} \left( |u|^2u\right).
\end{equation}}
{Hence, if $\psi\in {H}^1\cap \dot{H}^{\frac\gamma4,4}$ satisfies \eqref{original} then $u\in \dot{H}^{1-\frac\gamma4}\cap\dot{H}^{-\frac\gamma4}\cap L^4$ satisfies \eqref{original2}; and the converse implication also holds.} In this study, we focus on a numerical investigation of {a particular case of \eqref{transformed}, where (noticing  $D^2=-\partial_x^2$), $\beta=\gamma/2$, and $\tilde{\lambda}=-\lambda$} together with the  power-law  nonlinearity:
\begin{equation}\label{nonlocalNLS}
    \begin{cases}
        \ii u_t-\lambda u_{xx}=\zeta D^{\beta} (| u |^{2\sigma} u),
    \qquad (x,t)\in\rr{\times\rr^+},\\
    u(x,0)=u_0(x),
    \end{cases}
\end{equation}
where {$\zeta$ measures the nonlinearity strength and is often normalized to $=\pm 1$. The parameters $\sigma>0$ and $\lambda$ are nonzero real numbers. The constant $\zeta$ essentially   controls how the wave function interacts with itself due to nonlinearity,  and the behavior of the solutions strongly depends on its value/sign.}
If we normalize $\lambda$ to be $+1$, equation \eqref{nonlocalNLS} is referred to as {\it focusing} when $\zeta=+1$ and as  {\it defocusing}  when {$\zeta=-1$}.  The dependence of this model on the parameter $\beta$, makes it possible to explore different regimes of wave turbulence.  Equation \eqref{original} is a Hamiltonian system, with the Hamiltonian given by:
\begin{equation}
H(u)=\frac12\int_\mathbb{R} \left(\lambda|D^{\frac\alpha 2}u|^2-\frac{\zeta}{2}|D^{\frac{\gamma}{4}}u|^{4}\right)\dd x.
\end{equation}
As noted in \cite{zgpd}, equation \eqref{original} can be conveniently expressed in Fourier space as:
\[
\ii\frac{\hat{u}_k}{\partial t}
=\lambda |k|^{\alpha}
\hat{u}_k+\zeta\iiint T_{123k}\hat{u}_1\hat{u}_2\overline{\hat{u}}_3\delta(k_1+k_2-k_3-k)\dd k_1\dd k_2\dd k_3
\]
 where $\hat{u}_k=\hat{u}(k,t)$ denotes the $k$–th Fourier coefficient of $u$. In this formulation, \eqref{original} resembles the one-dimensional Zakharov equation, with the interaction coefficient given by:
 $$ T_{123k}=T(k_1,k_2,k_3,k)=\zeta\abso{k_1k_2k_3k}^{\beta /4}.
 $$

 From a mathematical standpoint, equation \eqref{nonlocalNLS} bears resemblance to the generalized Schrödinger equation
\[
\begin{cases}
    \ii u_t - \lambda u_{xx} = \scal{D}^\beta f(u,\bar{u}),
\\
u(x,0)=u_0(x),
\end{cases}
\]
where $\scal{\cdot}=1+|\cdot|$. This model has undergone extensive investigation, particularly in the case where $\beta = 1$ and $f(u,\bar u)=u^2$  (see \cite{ohst}). It is important to note that the corresponding equation is ill-posed, as evidenced by the phenomenon of norm inflation observed in the mapping $u_0 \mapsto u$, as discussed by Christ \cite{christ} and further elaborated upon in related literature \cite{related}.
In the work conducted by Stefanov \cite{stefanov}, the existence of weak solutions in $H^1$ was established, subject to an additional condition of smallness, specifically when $\sup_x \left| \int_{-\infty}^x u_0(y) \, dy \right| \ll 1$.
Furthermore, recent advances in local well-posedness have been achieved by Bejenaru \cite{bejenaru1, bejenaru2}, and Bejenaru and Tataru \cite{bejenaru-tataru}, demonstrating the existence of solutions, even for data that may not necessarily exhibit smallness, within weighted Sobolev spaces.

 To gain further insight, we first observe that the following quantities {(the total energy (or Hamiltonian), the mass of the wave function, and momentum, receptively)} are formally conserved   by the time evolution of \eqref{nonlocalNLS}:
\begin{eqnarray}
E(u)&=&\frac12\int_\rr\left(\lambda \abso{D^{-\frac\beta2}u_x}^2-\frac\zeta{\sigma+1}|u|^{2\sigma+2}\right) \dd x,\label{energy}  \\
F(u)&=&\int_\rr\abso{D^{-\frac\beta2}u}^2\dd x,\label{mass}\\
P(u)&=&\Im \scal{D^{-\frac\beta2}u,D^{-\frac\beta2}u_x}_{L^2}.\label{mass}
\end{eqnarray}
Inheriting from the above invariants, we can  define (see \cite{amin}) the space $\nx$ via the norm
$$
\|u\|_\nx^2=F(u)+\norm{D^{1-\beta/2}u}_\lt^2.
$$
Besides the conservation laws mentioned above,   equation \eqref{nonlocalNLS} is also invariant under the following scaling transformation:
\[
u(x,t)\mapsto u_\tau(x,t)=\tau^{\frac{2-\beta}{2\sigma}}u(\tau x,\tau^2 t)
\]
for any $\tau>0$. In other words, if \(u\) solves \eqref{nonlocalNLS}, then so does $u_\tau$. Consequently, under this scaling transformation, the homogeneous $\dot{H}^s (\rr)$ Sobolev norm of $u_\tau$ behaves as follows:
\[
\|u_\tau\|_{\dot{H}^s} \equiv \|D^s u_\tau\|_{L^2} = \tau^{-s+\frac{\sigma-2+\beta}{2\sigma}}\|u\|_{\dot{H}^s}.
\]
The equation is referred to as $\dot{H}^s$ critical whenever this scaling leaves the $\dot{H}^s$ norm invariant, that is whenever
\[
s=s_c=\frac{\sigma-2+\beta}{2\sigma}.
\]

When $\beta=0$,  \eqref{nonlocalNLS} reduces to the classical Schrödinger equation (NLS),
  \begin{equation}\label{NLS-0}
 \begin{cases}
        \ii u_t-\lambda u_{xx}=\zeta    | u |^{2\sigma} u ,\\
        u(x,0)=u_0(x),
 \end{cases}
    \end{equation}
    a canonical model for weakly nonlinear wave propagation in dispersive media. In this case the energy space $\nx$ turns into $H^1(\rr)$, and the mass conservation is $F(u)=\|u\|_{L^2}^2$. For $s_c = 0$, the corresponding mass critical case is found for $\sigma=2$.  Numerous results pertaining to the Cauchy problem associated with \eqref{NLS-0} exist (see, for example, \cite{linaresponce}). It is well known (see for example \cite{cazenave}) that in the mass subcritical case $\sigma < 2$, the classical NLS is globally well-posed, regardless of the sign of $\zeta$. On the other hand, finite time blow-up of solutions in the $\dot{H}^1 (\rr)$ can occur in the focusing case   as soon as $\sigma\geq2$. Moreover, it is known that for mass critical NLS, the threshold for finite time blow-up is determined by the mass of the corresponding ground state, that is the unique positive   solution $Q_0$ of the nonlinear elliptic equation
    \[
    \ff-\ff''=\ff^{2\sigma+1}.
    \]
    In other words, if $\sigma=2$ and $F(u_0)<F(Q_0)$, global existence still holds, whereas blow-up occurs as soon as $F(u_0) \geq F(Q_0)$.

 Regarding the local Cauchy problem, \eqref{nonlocalNLS} is trivially locally well-posed in $H^s$ with $s > 1/2 + \beta$, without the use of dispersive estimates, by employing the fractional pointwise inequality \cite{corma}. However, much better results are anticipated when leveraging the dispersive properties of the NLS group, as demonstrated in \cite{amin}. Specifically, in the case $\sigma = 1$ with $\beta < 1$, in \cite{amin} it was  established the local well-posedness of \eqref{nonlocalNLS} in $H^s(\mathbb{R})$ with $s > \beta/2$. Additionally, the flow map associated with the initial value problem   \eqref{nonlocalNLS} fails to be locally uniformly continuous when  $0 < \beta < 2/3$  and
 \[
\frac{3\beta - 2}{2(3 - 3\beta)}   < s <  \frac\beta2.
 \]

 This paper conducts numerical simulations to explore how the non-local operator $D^\beta$   affects various mathematical properties of the solutions \eqref{nonlocalNLS}. Roughly speaking, we focus on the interaction between dispersion and (non-local) nonlinearity in the time evolution of \eqref{nonlocalNLS}.  The split-step Fourier method is a quite efficient and popular numerical method for the well-known NLS equation \cite{weideman,fornberg1,pathria2,besse,muslu}. This method have the advantages
that it provides accurate solutions and it is unconditionally stable. For this aim, we propose the split-step Fourier method for the time evolution of the nonlocal NLS equation. Since no explicit solutions for standing wave solutions of equation \eqref{nonlocalNLS} are, except in the classical
case $\beta=0$, unknown, then some numerical method for the generation of approximate profiles is required. Therefore, we apply Petviashvili's iteration method to generate the standing wave solutions for $\beta \neq 0$.
Then, we study the dynamics of standing waves of \eqref{nonlocalNLS} by using the
split-step Fourier method. The nonlocal NLS equation also has a two-parameter family of standing wave solutions entitled boosted standing waves. We generate the boosted standing waves numerically by using Petviashvili's iteration method.
Although there has been lots of  numerical studies on one parameter
standing wave solutions for NLS  type equations, to the best our knowledge,
the numerical generation of a two-parameter family solutions is studied first in literature.  The stability of  boosted standing waves of \eqref{nonlocalNLS}  is meticulously studied numerically by checking the convexity of the Lyapanuv function and the long behavior of the  boosted standing wave solution
under small perturbations. We  also present some careful numerical
simulations of the semi-classical scaling of the focusing and defocusing nonlocal NLS equation. {Intuitively, we anticipate the model to exhibit better behavior for smaller values of $\beta$. Indeed, we will see that the critical index of stability is $\frac{2-
\beta}{1+\beta}$ where the ground states are stable if $\sigma<\frac{2-
\beta}{1+\beta}$ while they are unstable if $\sigma>\frac{2-
\beta}{1+\beta}.$ We also notice that the same index serves to obtain a dichotomy between the global solutions or the blow-up solutions. Compared with the classical NLS equation ($\beta=0)$, we observe that the instability of ground states or blow-up range  of $\sigma$ for \eqref{nonlocalNLS} is a decreasing function of $\beta$. See Theorems \ref{uniform} and \ref{stab-thm}.} Our analysis include investigating the specific nature of finite-time blow-up (such as whether it is self-similar), analyzing qualitative aspects of the associated ground state solutions (including their stability), and assessing the potential for well-posedness in the energy supercritical regime.

This paper is organized as follows. In Section \ref{Stability and Long time behavior}, we initially present some results about the existence of ground states of \eqref{nonlocalNLS}, and we generate  standing wave solutions numerically. Next, we derive the conditions under which the solutions remain in the energy space globally in time. Additionally, based on stability theory, we numerically investigate the stability of the standing wave solutions. In Section \ref{Boosted standing waves}, we demonstrate that the boosted standing waves exist under certain conditions, and then analyze the stability of these solutions for different wave speeds. Finally, in Section \ref{semi-classical limit}, we consider the semi-classical limit of \eqref{nonlocalNLS} in both focusing and defocusing cases.

\section{Stability and long time behavior}\label{Stability and Long time behavior}
In this section, we present mathematical findings relevant to the Cauchy problem linked with \eqref{nonlocalNLS}, serving as a foundation for our ensuing numerical simulations.

First, we report that the energy space $\nx$ is embedded into the Lebesgue space $L^q(\rr)$ under certain conditions. This connection will be crucial in the study of standing waves and global solutions of \eqref{nonlocalNLS}. {The proof of the following lemma is given in \cite{amin} for the case $-1/2<\beta<3/2$. However, the argument used there is also valid under the following assumptions.}

\begin{lemma} \label{embed}
	Let $-1<\beta<2$ and $\max\{0,\frac{-\beta}{1+\beta} \}\leq q\leq \frac{q^\ast}{2}-1$, where
	\[
	q^\ast=\begin{cases}
		\frac{2}{ \beta-1},&  \beta>1,\\
		\infty^-,&  \beta\leq1,
	\end{cases}\]
	where $\infty^-$ is any number $q_1<\infty$.
	Then there is a constant $C>0$   such that for any $g\in\nx$,
	\begin{equation}\label{gn}
		\|g\|_{L^{2q+2}(\rr)} \leq C
(F(g))^{\frac12- \frac14(\beta+\frac{q}{q+1})}
  \left\|D ^{ 1-\frac\beta2 }g \right\|_{L^2(\rr)}^{ \frac\beta2+\frac{q}{2(q+1)} }.
	\end{equation}
	As a consequence, it follows that $\nx$ is  continuously embedded in $L^{2q+2}(\rr)$.
\end{lemma}

\subsection*{Standing wave and ground state}
Here, we wonder whether the nonlocal NLS equation \eqref{nonlocalNLS} admits standing wave solutions of the form {$u(x,t)=\ee^{-\ii\omega t}\varphi (x)$,} where $\omega >0$ represents the standing wave frequency and the profile $\varphi$ is a real-valued time-independent function satisfying
\begin{equation}\label{stand-3}
  \omega \varphi-\lambda\varphi^{\prime\prime}=\zeta D^{\beta} ( |\varphi|^{2\sigma} \varphi).
\end{equation}


\begin{theorem}
Let $\lambda,\omega,\sigma>0$.  Then, Equation \eqref{stand-3} possesses no nontrivial solution
$\ff\in \nx\cap L^{2\sigma+2}(\rr)$ if $\zeta=-1$, or if $\zeta=+1$ and either $\beta\geq1+\frac1{\sigma+1}$ or $\beta\leq-\frac{\sigma}{\sigma+1}$.
\end{theorem}

\begin{proof}
The proof is followed from the same lines of one of \cite[Lemma 4.2]{amin} by using
the following Pohozaev identities
 \begin{equation}\label{poho-1}
\omega F(\ff)=\theta_0\norm{\ff}_{L^{2\sigma+2}}^{2\sigma+2}
,\qquad
 \norm{D^{1-\frac\beta2}\ff}_{L^2}^2=\theta_1\norm{\ff}_{L^{2\sigma+2}}^{2\sigma+2},
 \end{equation}
 where
 \begin{equation}\label{theta0}
     \theta_0=\frac{(\sigma+1)(1-\beta)+1}{2(\sigma+1)}
 \end{equation}
 and $\theta_1=1-\theta_0$; so we omit the details.
\end{proof}
The existence
of standing wave solution of \eqref{nonlocalNLS} with $\beta \in (-\frac{1}{2}, \frac{3}{2})$ and $\sigma=1$ for focusing case ($\zeta=1$) is discussed in \cite{amin}. Due to appearance of fractional derivatives, the existence of nontrivial solutions of \eqref{stand-3} can be also derived via  the maximization problem (see \cite{bfv})
\[
\sup_{0\not\equiv u\in \nx}\frac{\|u\|_{L^{2\sigma+2}}}{\|D^{-\frac\beta2}u\|_{L^2}^{1-\theta}\|D^{1-\frac\beta2}u\|_{L^2}^{\theta}},
\]
where $\theta=\frac\beta2+\frac{\sigma}{2(\sigma+1)}$. Similar to \cite{cvpde2018,amin}, we can show that
\[
 M_\omega=\inf\sett{S_\omega(u),\; u\in\nx,\;\scal{S'_\omega(u),u}=0}
\]
with $S_\omega=E+\frac\omega2 F$, attains a minimum that is (up to a scaling) a nontrivial solution of \eqref{stand-3}. Indeed, we can first show that there exists a nontrivial solution of \eqref{stand-3} which is derived by finding the critical points of $S_\omega$ enjoying the mountain-pass geometry. This is strongly connected to show that the Palais-Smale sequence is non-vanishing by following an argument similar to \cite[Lemma 2.14]{cvpde2018}. Moreover, these critical points are ground state at the same Mountain Pass level $M_\omega$.  We also observe that $M_\omega$ is independent of critical points if there is more than one critical point.

 Furthermore, if $\beta\leq0$, there exists an even, strictly
positive, decreasing solution of \eqref{stand-3}.
It is worth noting from \eqref{stand-3} that any solution of \eqref{stand-3} with $\beta>0$ is sign-changing.
\begin{theorem}\label{gs-thm}
Let $\lambda,\omega,\sigma>0$ and $-1<\beta<2$.  Then, \eqref{stand-3} possesses a ground state if
\[
\max\sett{0,\frac{-\beta}{1+\beta} }< \sigma<  \begin{cases}
     \frac{2-\beta }{\beta-1},&\beta>1,\\
    \infty,&\beta\leq1.
\end{cases}
\]
\end{theorem}

\subsection{Numerical generation of standing wave solutions}

Now, we study the form of the standing wave solutions numerically.
By using the scaling $\ff(x)=\omega^{\frac{2-\beta}{4\sigma}}\tilde\ff(\omega^{\frac12}x)$, we can assume that $\omega=1$ in \eqref{stand-3}.
  Since we do not know the explicit standing wave solutions
 of equation \eqref{nonlocalNLS} for any nonzero $\beta$, we first use the Petviashvili iteration method \cite{pelinovsky, petviashvili, yang} to generate the {standing} wave solutions numerically. Petviashvili's iteration method was first introduced by V.I. Petviashvili for the  Kadomtsev-Petviashvili equation
in \cite{petviashvili} to generate a solitary wave solution. Applying the Fourier transform
to equation \eqref{stand-3} for $\omega=1, \zeta=1$ yields
\begin{equation}
  (1+\lambda k^2) \widehat{Q}= |k|^{\beta}  \widehat{ (|Q|^{2\sigma} Q)},  \label{groundhat}
\end{equation}
where $\widehat{Q}$ is the Fourier transform of the approximation to the profile $\varphi$. The standard iterative algorithm is given in the form
\begin{equation}
\widehat{Q}_{n+1}(k) = \frac{|k|^{\beta} \widehat{ (|Q_n|^{2\sigma} Q_n)}}{1+\lambda k^2}.
\end{equation}
The main idea of the Petviashvili method is to add a stabilizing factor into the fixed-point iteration. Therefore, we avoid the iterated solution converging to zero solution or diverging. The Petviashvili method  for equation \eqref{stand-3} is given by
\begin{equation}
\widehat{Q}_{n+1}(k) = (M_n)^{\nu} \frac{|k|^{\beta} \widehat{ (|Q_n|^{2\sigma} Q_n)}}{1+ \lambda k^2}
\label{petviashvilischeme}
\end{equation}
with
\begin{equation*}
  M_n=\frac{\int_{\mathbb{R}} [(1+\lambda k^2) [\widehat{Q}_{n}(k)]^2 \dd k }
  {\int_{\mathbb{R}}  {|k|^{\beta} \widehat{ (|Q_n|^{2\sigma} Q_n)}} \widehat{Q}_{n}(k) \dd k },
\end{equation*}
for some parameter $\nu$. Following \cite{pelinovsky}, we choose $\nu=\displaystyle\frac{2\sigma+2}{2\sigma+1}$ to provide fastest convergence.

\begin{figure}[h!bt]
 \begin{minipage}[t]{0.45\linewidth}
   \includegraphics[width=3.6in]{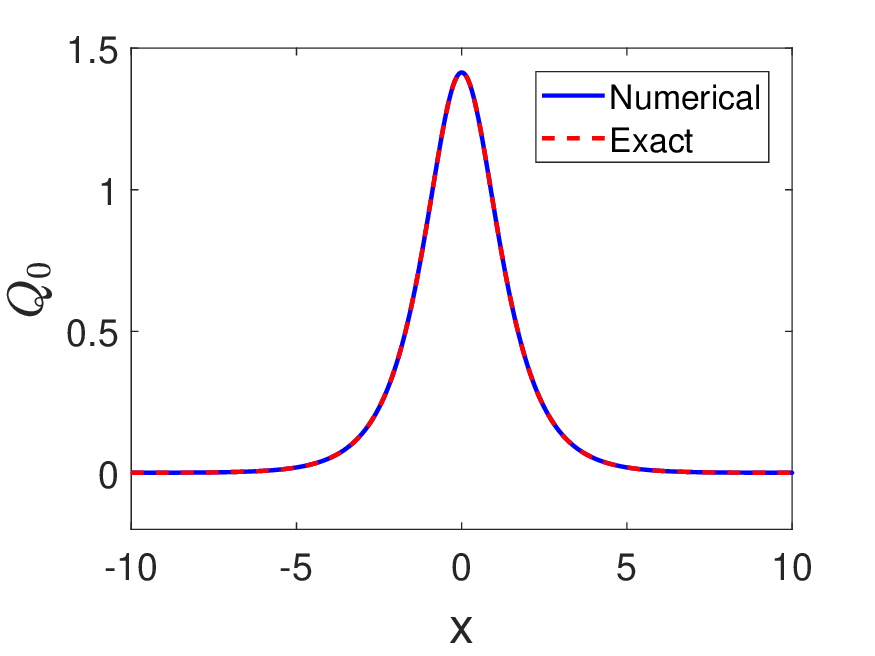}
 \end{minipage}
\hspace{30pt}
\begin{minipage}[t]{0.45\linewidth}
   \includegraphics[width=3.6in]{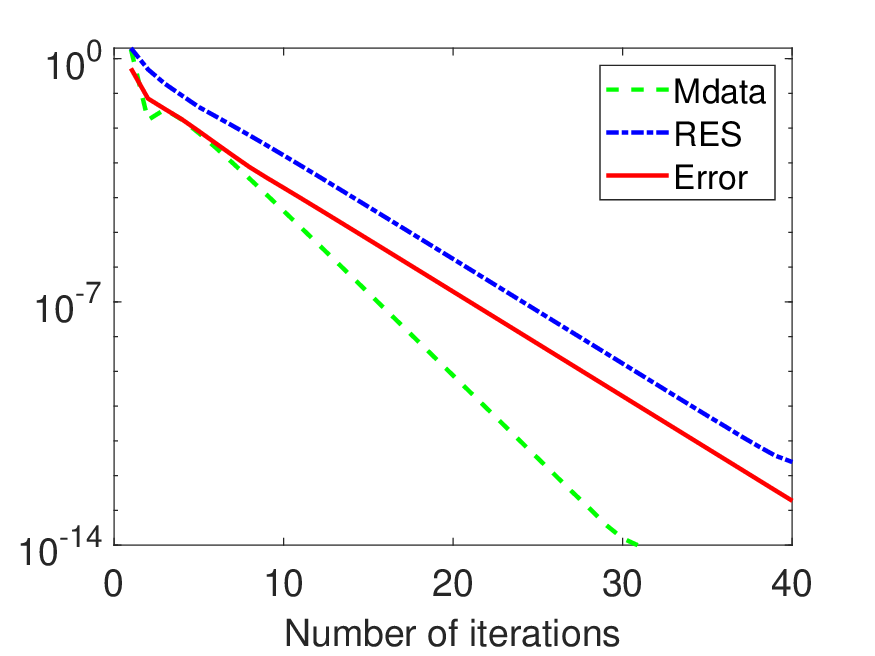}
 \end{minipage}
  \caption{Both exact and generated standing wave solutions of focusing NLS  equation ($\beta=0$)  on the interval $[-10,10]$ and the variation of the $Error(n)$, $|1-M_n|$ and $RES$ with the number of iterations in semi-log scale.}
 \label{nls}
\end{figure}

The overall iterative process is  controlled by the  error,
\begin{equation}
  Error(n)=\|Q_n-Q_{n-1}\|_\infty,~~~~n=0,1,.... \nonumber
\end{equation}
 between two consecutive iterations defined with the  number of iterations,  the stabilization factor error
\begin{equation}
|1-M_n|, ~~~~n=0,1,.... \nonumber
\end{equation}
and the residual error
\begin{equation}
{RES(n)}= \|{\cal S} Q_n\|_\infty, ~~~~n=0,1,.... \nonumber
\end{equation}
where
\begin{equation}
{\cal S}Q= (1+ pk^2) \widehat{Q}- |k|^{\beta} \widehat{ Q^{2\sigma+1}}. \nonumber
\end{equation}

\begin{figure}[h!]
 \begin{minipage}[t]{0.45\linewidth}
\includegraphics[width=3.6in]{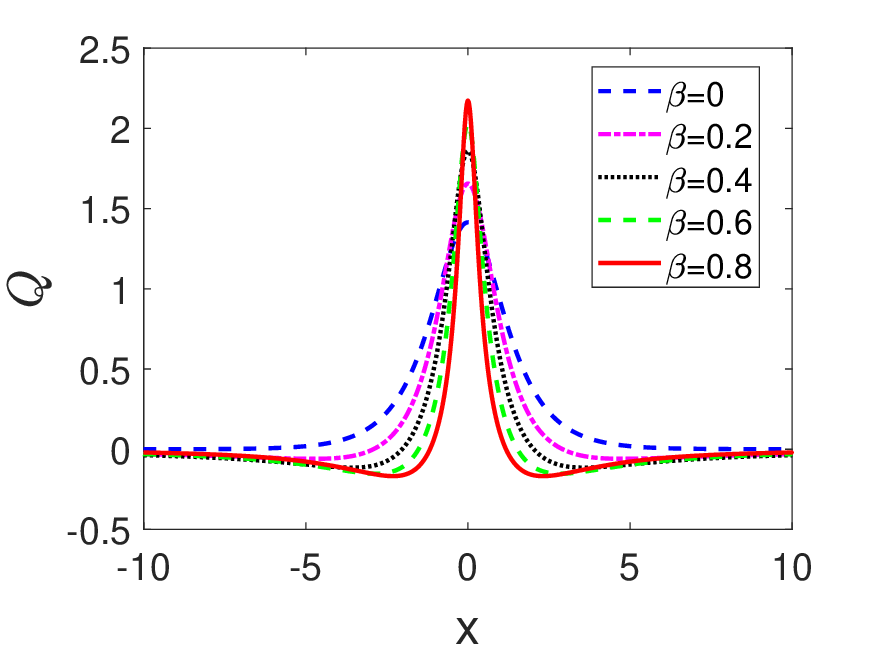}
 \end{minipage}
\hspace{30pt}
\begin{minipage}[t]{0.45\linewidth}
   \includegraphics[width=3.6in]{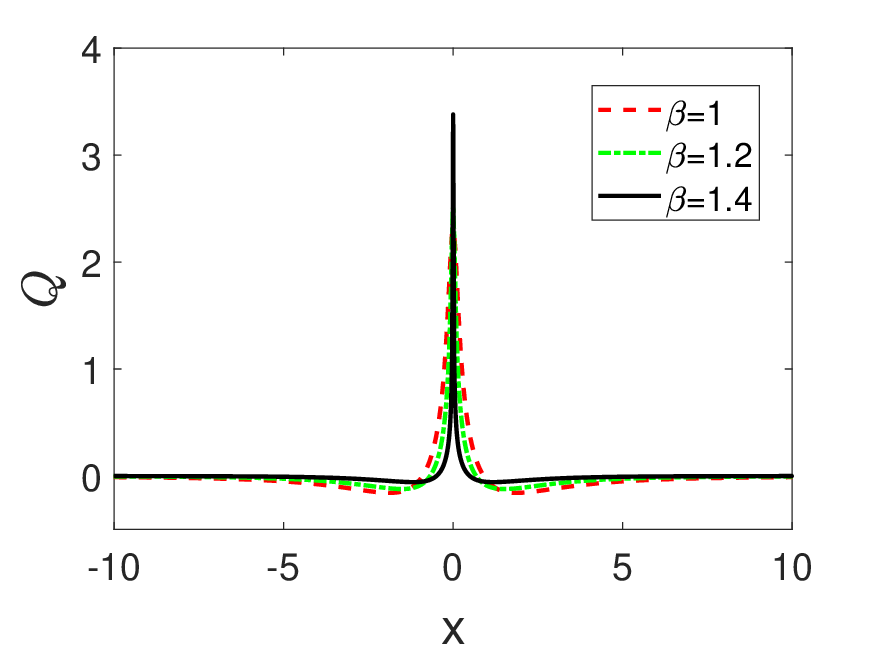}
 \end{minipage}
 \hspace{30pt}
  \caption{Computed standing wave solutions of \eqref{stand-3} for several values of $\beta\in [0,\frac{3}{2})$ and $\omega=1$.}
 \label{groundstate-nonlocalNLS1}
\end{figure}

\begin{figure}[h!]
\centering
\includegraphics[width=3.6in]{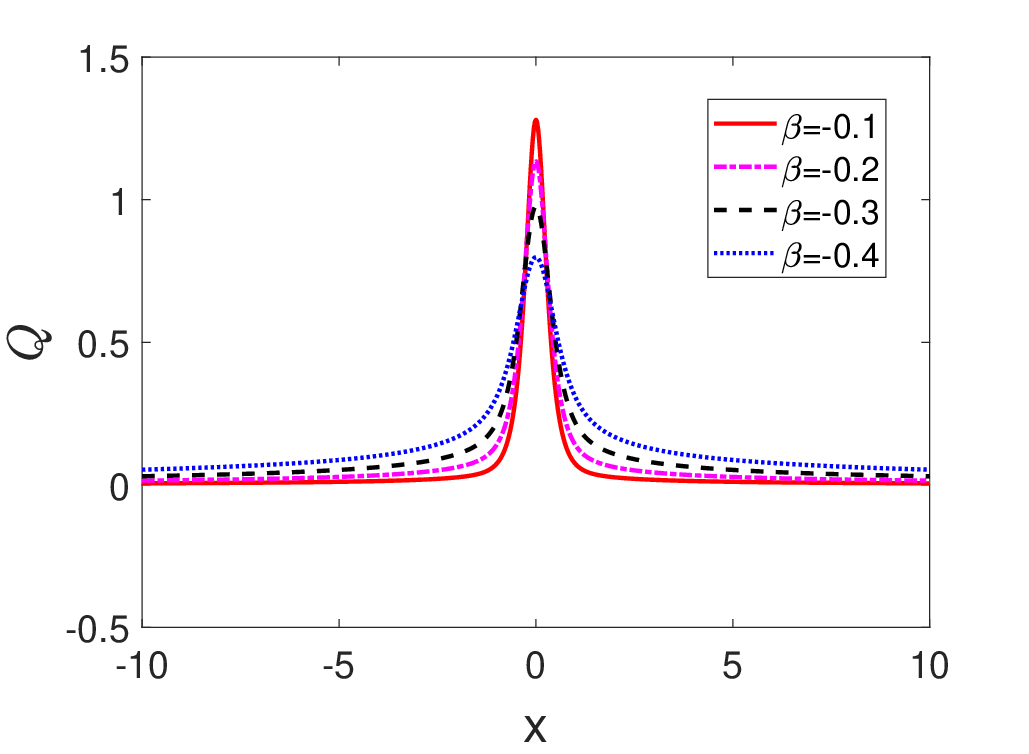}
\caption{Computed standing wave solutions of \eqref{stand-3} for several values of $\beta \in (-\frac{1}{2}, 0)$ and $\omega=1$.}
\label{groundstate-nonlocalNLS2}
\end{figure}

Introducing the time-reverse transformation $v(x,T)=u(x,-t)$, equation \eqref{nonlocalNLS} is equivalent to
\begin{equation}
     \ii v_T+\lambda v_{xx}=-\zeta D^{\beta} (| v |^{2\sigma} v) \label{NLS}]
\end{equation}
This equation turns out to be well-known focusing Schrödinger equation for $\lambda=1,\zeta=1$, and $\beta=0$.
The standing wave solutions are of the form $v(x,T)=\ee^{ \ii \omega T} \Phi(x)$, where $\omega>0$. Equation \eqref{NLS} is
satisfied if and only if $\Phi$ is the solution of the stationary equation
\begin{equation}
-\omega \Phi+\Phi^{\prime\prime}= - |\Phi|^{2\sigma} \Phi.
\end{equation}

\begin{figure}[h!]
		\begin{minipage}[t]{0.45\linewidth}
			\centering
			\includegraphics[height=5.5cm,width=7.5cm]{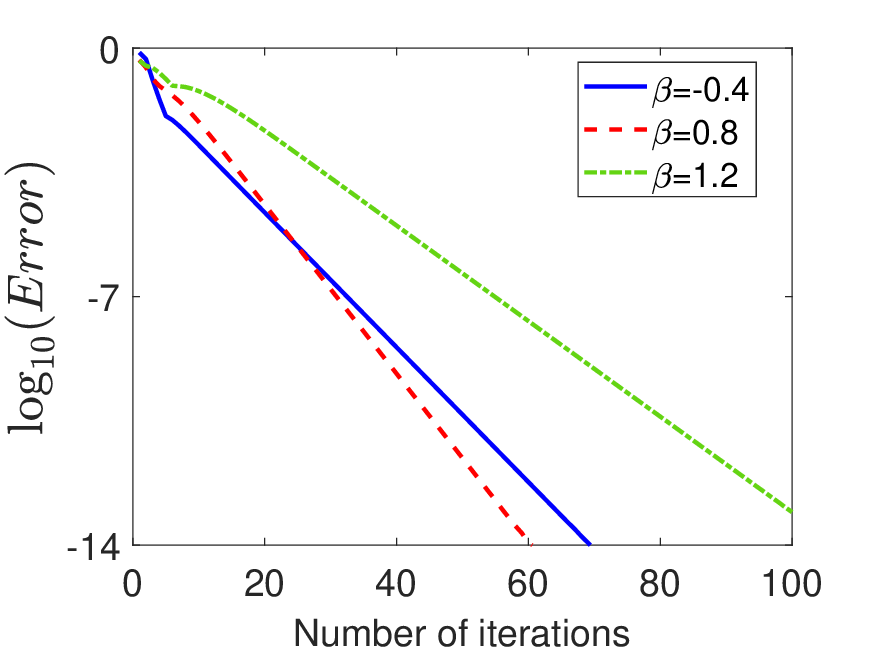}
		\end{minipage}%
		\hspace{20pt}
		\begin{minipage}[t]{0.45\linewidth}
			\centering
			\includegraphics[height=5.5cm,width=7.5cm]{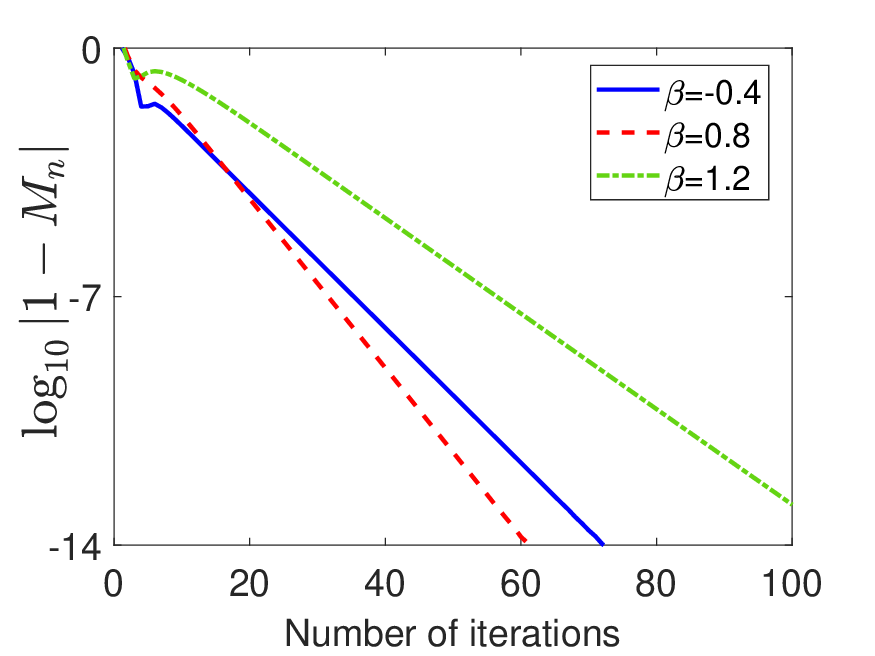}
		\end{minipage}
		\begin{minipage}[t]{0.45\linewidth}
			\centering
			\includegraphics[height=5.5cm,width=7.5cm]{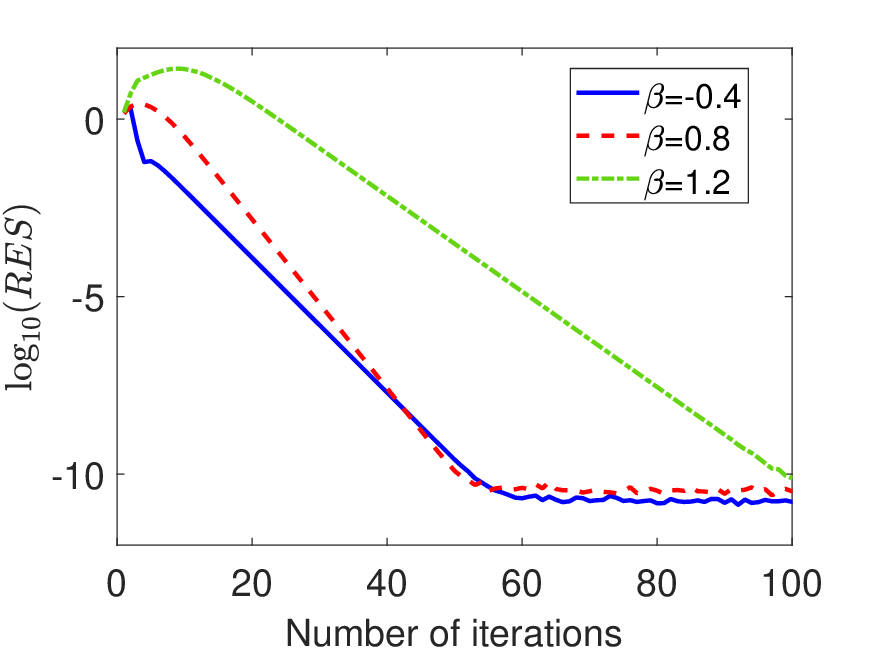}
		\end{minipage}%
		\hspace{40pt}
		\begin{minipage}[t]{0.45\linewidth}
			\centering
			\includegraphics[height=5.5cm,width=7.5cm]{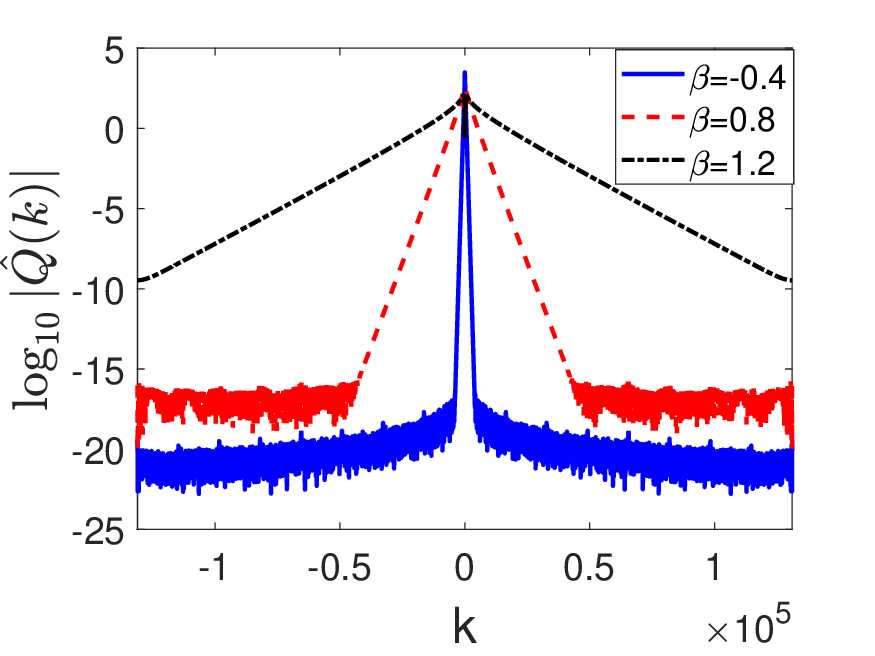}
	\end{minipage}
		\caption{ The variation of three different errors with the number of iterations and the modulus of the Fourier coefficients for the standing wave profiles correspond to $\beta=-0.4, \beta=0.8$ and $\beta=1.2$.}\label{errors}
	\end{figure}

\noindent
Its exact standing wave  solution  is well-known (see e.g. \cite{sulem})
\begin{equation}\label{exact}
    {Q_0}(x)=\frac{(\sigma +1)^{\frac1{2\sigma}}}{\cosh^{\frac1\sigma} (\sigma x)} .
\end{equation}
  We choose the space interval as  $x \in [-1000,1000] $ taking the number of grid points as $N=2^{18}$.
\noindent
In Figure \ref{nls}, we present both numerical and exact standing wave solution of \eqref{NLS} on the interval $[-10,10]$ to see more visible and
the variation of three different errors
with the number of iterations in semi-log scale.  As it is seen from Figure \ref{nls}, our proposed numerical scheme captures the exact standing wave solution  for $\beta=0$ remarkably well.
 $L^{\infty}-$error between numerical and exact standing wave solution is also about $10^{-14}$.

\noindent
For nonlocal NLS equation with nonzero $\beta \in (-\frac{1}{2}, \frac{3}{2})$ for focusing case ($\zeta=\lambda=\sigma=1$), no explicit solutions of \eqref{stand-3} are known. Computed standing wave solutions of
\eqref{stand-3} for several values of $\beta\in [0, \frac{3}{2})$ and $\omega=1$ are depicted in Figure \ref{groundstate-nonlocalNLS1}.

The generation of the standing wave solutions of the nonlocal NLS equation with $\beta \in (-\frac{1}{2}, 0)$ is numerically challenging. The discontinuity
of the  Fourier multiplier in \eqref{groundhat},   i.e.  $|k|^{\beta}  \widehat{ Q^{2\sigma+1}}(k)$, at $k=0$ is resolved
by commonly-used technique of setting $|k|^{\beta}$ as $0$. Computed standing wave solutions of
\eqref{stand-3} for several values of $\beta \in (-\frac{1}{2}, 0)$ and $\omega=1$ are illustrated in Figure \ref{groundstate-nonlocalNLS2}.
Figures  \ref{groundstate-nonlocalNLS1} and \ref{groundstate-nonlocalNLS2} show that the standing wave profiles become more peaked
with the larger values of $\beta$.  In Figure \ref{errors}, we show the variation of three different errors with the number of iterations and the modulus of the Fourier coefficients for the standing wave profiles correspond to $\beta=-0.4, \beta=0.8$ and $\beta=1.2$. The numerical results confirm the convergence of the iterative scheme \eqref{petviashvilischeme}.  The accuracy of the
approximation in space is also controlled by the Fourier coefficients since the numerical error is of the order of the Fourier coefficients
for the highest wave numbers. It can be seen that the Fourier coefficients decrease
 to machine precision for the high wavenumbers when $\beta=-0.4, \beta=0.8$, whereas they decrease to $10^{-9}$ when $\beta=1.2$.
This shows that the solution is numerically well resolved.

\subsection{Uniformly bounded   solutions}
Next, we investigate the long time behavior of the solutions of \eqref{nonlocalNLS}. The following theorem give the conditions under which the solutions of \eqref{nonlocalNLS} are uniformly bounded in the energy space $\nx$.
\begin{theorem}\label{uniform}
	Let $\lam>0$, $u_0\in \nx$, {$-1<\beta<2$},  and $u\in C([0,T);\nx)$ be the
	solution of \eqref{nonlocalNLS}, associated with the initial value $u_0$, and $Q$ be a ground state of \eqref{stand-3}.  Then
	$u(t)$ is uniformly bounded in $\nx$, for $t\in [0,T)$, if  $\zeta=-1$, or $\zeta=+1$  and   one of the
	following cases occurs:
	\begin{enumerate}[(i)]
		\item $\sigma<\frac{2-\beta}{1+\beta}$;
  \item $\sigma=\frac{2-\beta}{1+\beta}$ and
  \[
  F(u_0)< \lam ^\frac{1+\beta}{2-\beta}F(Q);
  \]
  \item $\sigma>\frac{2-\beta}{1+\beta}$, and $u_0$ satisfies
  \[
       \|D ^{ 1-\frac\beta2 }u_0\|_\lt^{\beta(\sigma+1)+\sigma-2}
      (F(u_0))^\frac{(1+\sigma)(1-\beta)+1}{2}
      <\lam
      \|D ^{ 1-\frac\beta2 }Q\|_\lt^{\beta(\sigma+1)+\sigma-2}
      (F(Q))^\frac{(1+\sigma)(1-\beta)+1}{2}
      \]
      and
      \[
          (E(u_0))^{\frac{\beta(\sigma+1)+\sigma-2}{2}}  (F(u_0))^\frac{(1+\sigma)(1-\beta)+1}{2}
      <
      \lam^{\frac{\beta(1+\sigma)+\sigma}{2}}
            (E(Q))^{\frac{\beta(\sigma+1)+\sigma-2}{2}}(F(Q))^\frac{(1+\sigma)(1-\beta)+1}{2}.
  \]
	\end{enumerate}
\end{theorem}





To prove   Theorem \ref{uniform}, we resort to the following lemma (\cite[Lemma 7.7.4]{cazenave}).
\begin{lemma}\label{cazenave}
Let $a,b > 0$ and $p > 1$. Assume that $b$ is small enough so that
the function $f (x) = a - x + bx^p$ is negative for some $x>0$, and let $x_0$ be the first
(positive) zero of $f$. Let $I\subset\rr$ be an interval and let $\phi\in C(I,\rr^+)$ satisfy
    \[
    \phi(t)\leq a+b\phi(t)^p\qquad \forall t\in I.
    \]
If $\phi(t_0)\leq x_0$ for some to $t_0\in I$, then $\phi(t)\leq x_0$ for
all $t\in I$.
\end{lemma}

\begin{proof}[Proof of Theorem \ref{uniform}]
For simplicity, let us assume $\lambda=1$. The case $\zeta=-1$ is an immediate application of the invariance $E$. In the case $\zeta=+1$, applying Lemma \ref{embed} yields
    \[
   \begin{split}
       2E(u_0)&= 2E(u)\geq \phi(t)-\frac{C^{2\sigma+2}}{ \sigma+1 }(F(u))^{\frac\sigma2+1-(\frac{\sigma+1}{2})\beta}
(\phi(t)) ^
  { \frac{(\sigma+1)\beta+\sigma}{2} }\\
  &\geq \phi(t)-\frac{C^{2\sigma+2}}{ \sigma+1 }(F(u_0))^{\frac\sigma2+1-(\frac{\sigma+1}{2})\beta}
(\phi(t)) ^
  { \frac{(\sigma+1)\beta+\sigma}{2} },
   \end{split}
    \]
    where $\phi(t)=\norm{D^{1-\frac\beta2}u(t)}_{\lt}^2$.  On the other hand, following the arguments in \cite{zamp}, it can be shown that the optimal constant $C$ in \eqref{gn} satisfies
 \begin{equation}\label{bestc}
        C^{-1}=\theta_0^{\frac{\theta_0}{2}}\theta_1^{\frac{\theta_1}{2}}\|\ff\|_{L^{2\sigma+2}}^{2\sigma+2}
  \end{equation}
    where $\theta_0$ is the same as in \ref{theta0},
$\theta_1=1-\theta_0$, and $\ff$ is a ground state of
\eqref{stand-3}.
The   Pohozaev identities
 \eqref{poho-1}
show that
\[
E(\ff)=\frac{\theta_1(\sigma+1)-1}{2(\sigma+1)}\|\ff\|_{L^{2\sigma+2}}^{2\sigma+2}.
\]
    By applying Lemma \ref{cazenave} with $a=2E(u_0)$, $p= \frac{(\sigma+1)\beta+\sigma}{2}$, and
    \[
    b=\frac{C^{2\sigma+2}}{\sigma+1}(F(u_0))^{\theta_1(\sigma+1)},
    \]
  the results is deduced.
\end{proof}
\begin{remark}    \label{blow-up-theo}
 Let $\lam=\zeta=1$,   $\sigma>\frac{2-\beta}{1+\beta}$,   $u_0\in\nx$,   and $u\in C([0,T);\nx)$ be the
	solution of \eqref{nonlocalNLS}, associated with the initial value $u_0$. When $\beta=0$, it was demonstrated in \cite{glass} that for any negative initial data in $\nx=H^1(\rr)$ has finite variance, i.e. satisfying $xu_0\in L^2(\rr)$, the corresponding solution of
\eqref{NLS-0}
  enjoys the Variance-Virial Law
 \[
 \frac14\frac{\dd}{\dd t}\|xu(t)\|_{L^2}^2=\Im\scal{u(t),x\cdot \nabla u}_{L^2},
 \]
 and blows up in finite time when $\sigma>2$. The existence of blow-up solutions for negative radial data in space dimensions $n\geq2$ and for negative data in the one-dimensional case was established in \cite{ot-1,ot-2}. Additionally, it was shown in \cite{Hol-R} that in the mass and energy intercritical case, if initial data has nonnegative energy and satisfies the following inequalities:
 \begin{equation}\label{hol-r}
\begin{split}
    &\|\nabla u_0\|_{L^2}^{s_{c,n} }
     \|u_0\|_{L^2}^{1-s_{c,n}}
      >
   \|\nabla R\|_{L^2}^{s_{c,n} }
     \|R\|_{L^2}^{1-s_{c,n}},
 \\
    &(E(u_0))^{s_{c,n}}  (F(u_0))^ {1-s_{c,n}}
    <
    (E(R))^{s_{c,n}}(F(R))^{1-s_{c,n}},
\end{split}
\end{equation}
and if, in addition, $xu_0\in L^2(\rr^n)$ or $u_0$ is radial with $n\geq2$ and $s_{c,n}>0$, then the corresponding solution of
 \begin{equation}\label{nls-n}
 \ii u_t-\Delta u=|u|^{2\sigma}u,\quad x\in \rr^n
 \end{equation}
 blows up in finite-time. Here, $s_{c,n}=\frac n2-\frac1\sigma$ and $R$ is the ground state of \eqref{nls-n}, which is the unique (up to symmetries) positive radial solution of the elliptic equation
\begin{equation}\label{ell}
    \Delta R-R+|R|^{2\sigma}R=0.
\end{equation}
In the case $n=1$, see \eqref{exact}. Holmer and Roudenko \cite{Hol-R-1} demonstrated that if the initial data belongs to $H^1$ (not necessarily possessing finite variance or radial symmetry) and satisfies \eqref{hol-r}, then the corresponding solution exhibits one of two behaviors: it either blows up in finite time or it blows up over an infinite time   in the sense that there exists a sequence of times $t_n \rightarrow +\infty$ such that $\|\nabla u(t_n)\|_{L^2} \rightarrow \infty$. In \cite{dwz}, the authors  extended the findings of \cite{Hol-R-1} by establishing a blow-up criterion for \eqref{NLS-0} with initial data (without finite-variance and radial symmetry assumptions) in both the energy-critical and energy-supercritical regimes.
Although, due to presence nonlocal operator $D^\beta$, it is not easy to show a similar result for \eqref{nonlocalNLS}, one can formally show, using the optimal constant \eqref{bestc}, under the above technical assumptions that if either $E(u_0)<0$  or if  $E(u_0)\geq0$ and
 \[
       \norm{D ^{ -\frac\beta2 }(u_0)_x}_\lt^{\beta(\sigma+1)+\sigma-2}
      (F(u_0))^\frac{(1+\sigma)(1-\beta)+1}{2}
      >
      \norm{D ^{ -\frac\beta2 }Q_x}_\lt^{\beta(\sigma+1)+\sigma-2}
      (F(Q))^\frac{(1+\sigma)(1-\beta)+1}{2}
      \]
      and
      \[
          (E(u_0))^{\frac{\beta(\sigma+1)+\sigma-2}{2}}  (F(u_0))^\frac{(1+\sigma)(1-\beta)+1}{2}
      <
                   (E(Q))^{\frac{\beta(\sigma+1)+\sigma-2}{2}}(F(Q))^\frac{(1+\sigma)(1-\beta)+1}{2},
  \]
  where $Q$ is a ground state of \eqref{stand-3}.
 Then one
of the following statements holds:
\begin{itemize}
    \item $u(t)$ blows up in finite time, i.e. $T<+\infty$ and
    \[
    \lim_{t\uparrow T}\norm{D^{-\frac\beta2}u_x(t)}_\lt=+\infty;
    \]
    \item $u(t)$ blows up in infinite time  and
    \[
    \lim_{t\to \infty}\norm{D^{-\frac\beta2}u_x(t)}_\lt=+\infty;
    \]
\end{itemize}
\end{remark}
 \subsection{Stability}\label{stab-sub}
 Studied the existence of ground states for \eqref{nonlocalNLS}, it is natural to investigate the dynamical behavior of such solutions. Since \eqref{nonlocalNLS} can be written in the form of Hamiltonian dynamical system
 \[
 \ii u_t=E'(u),
 \]
 we recall the following definition for the stability of such systems.
 \begin{definition}
We say that a set $\mathcal{J}\subset\nx$ is stable with respect to  the Cauchy problem associated with \eqref{nonlocalNLS} if for any
$\epsilon>0$ there exists some $\delta>0$ such that, for any $u_0\in B_\delta(\mathcal{J})$, the solution $u$ of \eqref{nonlocalNLS}  with $u(0) = u_0$ satisfies $u(t)\in B_\epsilon(\mathcal{J})$ for all $t > 0$, {where
\[
B_\delta(\mathcal{J})
=\sett{v\in X,\;\inf_{z\in\mathcal{J}}\|v-z\|_X<\delta}.
\]}
Otherwise, we say $\mathcal{J}$ is unstable.
\end{definition}

If $\lambda,\omega,\sigma,\zeta>0$ and $\beta=0$. It is well-known that a positive solution $\ff$ of \eqref{stand-3} is ground state, that is $\ff$ minimizes the energy. In \cite{caselions}, Cazenave and Lions proved that the standing wave solution $\ee^{-\ii\omega t}\ff(x)$ is stable when $\sigma<2$, while Berestycki and  Cazenave in\cite{berecase} showed that it is unstable if $\sigma\geq2$. Grillakis, Shatah, and Strauss developed an abstract theory and gave a necessary and sufficient conditions for the stability of stationary states of Hamiltonian systems under certain assumptions on the spectrum of the linearized operator associated with \eqref{ell} $\ee^{-\ii\omega t}\ff(x)$ is stable (resp. unstable) when the Lyapunov function $\mathbb{d}(\omega)=S(\ff_\omega)=E(\ff_\omega)+\frac{\omega}{2}F(\ff_\omega)$ is  strictly increasing (resp. decreasing). However, it seems difficult to check the spectral properties of the linearized operator of \eqref{stand-3}, by using the ideas of \cite{physicad2021,narwa2022,elm} and the stability theory developed in \cite{gss}, we can show a weak version of stability.

Let $\ngg_\omega$ be denoted the set of ground states of \eqref{stand-3}.
Take $\ff_\omega\in\ngg_\omega$ and define  the Lyapunov function $\mathbb{d}(\omega)=S(\ff_\omega)$. The above argument shows that $\mathbb{d} (\omega)$ depends only $\omega$. An immediate corollary (see \cite{caselions}) of the minimization $M_\omega$  and Theorem \ref{gs-thm} is the stability of $\ngg_\omega$. Since the proof of the following theorem is analogous to Theorem 4.1 and Corollary 5.1 in \cite{narwa2022}, we omit the details.


\begin{theorem}\label{stab-thm}
{Let $-1<\beta<2$.} The set $\ngg_\omega$ is stable if $\mathbb{d}''(\omega)>0$, and it is unstable if  $\mathbb{d}''(\omega)<0$.
\end{theorem}

Since the nonlinearity of \eqref{nonlocalNLS} is of power type, so  the scaling $\ff_\omega(\cdot)=\omega^{\frac{2-\beta}{4p}}\ff_1(\omega^{\frac12}\cdot)$ provides a family of solutions of
\eqref{stand-3} for $\omega>0$ such that
\[
\mathbb{d}''(\omega)=\frac{((1-\beta)2\sigma+2-\beta)((2-\beta)(\sigma+1)-3\sigma)}{8(\sigma+1)}\omega^{\frac{(1-\beta)2\sigma+2-\beta}{2\sigma}-2}\|\ff_1\|_{L^{2\sigma+2}}^{2\sigma+2}>0
\]
if $\sigma<\frac{2-\beta}{1+\beta}$, while $\mathbb{d}''(\omega)<0$ if $\sigma>\frac{2-\beta}{1+\beta}$.

\subsubsection{Numerical study of stability of standing wave solutions }
In this section, we study the stability of the  generated standing wave solutions of the nonlocal  NLS equation by using  the split-step Fourier method.
Our aim is to fill the gap given in the above theorems.
The main idea of the split-step method is to decompose the original problem into subproblems that are simpler than the original problem and then to compose the approximate solution of the original problem by using the exact or approximate solutions of the subproblems in a given sequential order. To solve the nonlocal NLS equation numerically, we rewrite equation \eqref{nonlocalNLS} in the form
\begin{equation}
   u_t=({\cal L}+{\cal N})u =- \ii \lambda u_{xx}- \ii \zeta D^{\beta} (| u |^{2\sigma} u), \label{evolution}
\end{equation}
where ${\cal L}$ and ${\cal N}$ are linear and nonlinear operators, respectively.
We take a  finite interval $(a,b)$ of big enough length and assume that $u(x,t)$ satisfies the periodic boundary condition
$u(a,t)=u(b,t)$  for $t\in [0,T]$.
The interval $[a,b]$ is divided into $N$ equal subintervals with spatial mesh size
$h=(b-a)/N$, where the positive integer $N$ is even. The spatial grid points are given by
$x_{j}=a+ j h$,  $j=0,1,2,...,N$. The time step
 is denoted by $\tau$ and $\tau = T/M$ for some $M\in \mathbb{Z}^+$. The approximate solution to
$u(x_{j},t_n)$ is denoted by $u_{j}^n$. Since we discretize \eqref{evolution} by the Fourier
 spectral method, $u_{j}^n$ and its Fourier transform satisfy the following relations:
\begin{equation*}
    \hat{U}_{k}^n={\cal F}_{k}[u_{j}^n]=
          \frac{1}{N}\sum_{j=0}^{N-1}u_{j}^n\exp(- \ii kx_{j}),
           ~~~~-\frac{N}{2} \le k \le \frac{N}{2}-1~  \label{dft}
\end{equation*}
and
\begin{equation*}
    u_{j}^n={\cal F}^{-1}_{j}[\hat{U}_{k}^n]=
          \sum_{k=-\frac{N}{2}}^{\frac{N}{2}-1}\hat{U}_{k}^n\exp( \ii kx_{j}),
          ~~~~j=0,1,2,...,N-1     ~.         \label{invdft}
\end{equation*}
Here $\cal F$ denotes the discrete Fourier transform and
${\cal F}^{-1}$ its inverse.  These transforms are  efficiently computed
using a fast Fourier transform (FFT) algorithm. By using a spectral approach, a Gaussian-like quadrature of the  fractional Fourier transform is also derived in \cite{campos}. Equation \eqref{evolution} can be split into the linear equation
\begin{equation}
   u_t={\cal L}u=- \ii \lambda u_{xx} \label{linear}
\end{equation}
and nonlinear equation
\begin{equation}
   u_t={\cal N}u=- \ii \zeta D^{\beta} (| u |^{2\sigma} u). \label{nonlinear}
\end{equation}
The linear equation  \eqref{linear} is solved using the
discrete Fourier transform and the advancements in time are performed according to
\begin{equation*}
    u_{j}^{n+1}={\cal F}^{-1}_{j}[\exp( \ii {\lambda}k^{2}\tau)
                    {\cal F}_{k}[u_{j}^{n}]]. \label{gnlst}
\end{equation*}
In the Fourier domain, equation \eqref{nonlinear} can be written as
\begin{equation}
   \hat{u}_t=- \ii \zeta ~|k|^{\beta} \widehat{(| u |^{2\sigma} u)}; \label{nonlinearFourier}
\end{equation}
which was solved numerically by the fourth-order Runge-Kutta scheme. In practical computation, from time $t = t^n$ to $t = t^{n+1}$, we apply the fourth-order  splitting steps via a standard Strang splitting  \cite{mclachlan1} as

\begin{equation*}
  u( t)=\varphi_{2}(\omega  t)
                           \varphi_{2}[ (1-2\omega)  t ]
                           \varphi_{2}(\omega  t),
     \label{sp4}
\end{equation*}
where
\begin{equation*}
 \varphi_{2}( t)=\exp\paar{{1 \over 2}t {\cal N}}
                           \exp \paar{t {\cal L}}
                           \exp\paar{{1 \over 2} t {\cal N} }
\end{equation*}
and $\omega=(2+2^{1/3}+2^{-1/3})/3$.  {As usual for explicit time discretization schemes, the numerical stability is conditional
and guaranteed only under the  Courant–Friedrichs–Lewy (CFL) condition.  We choose a very fine time step to satisfy the CFL condition for stability  as   $\tau < C h^{2}$  with $C$ a positive constant independent of $h$.
}

\begin{figure}[h]
 \begin{minipage}[t]{0.4\linewidth}
   \includegraphics[width=3in]{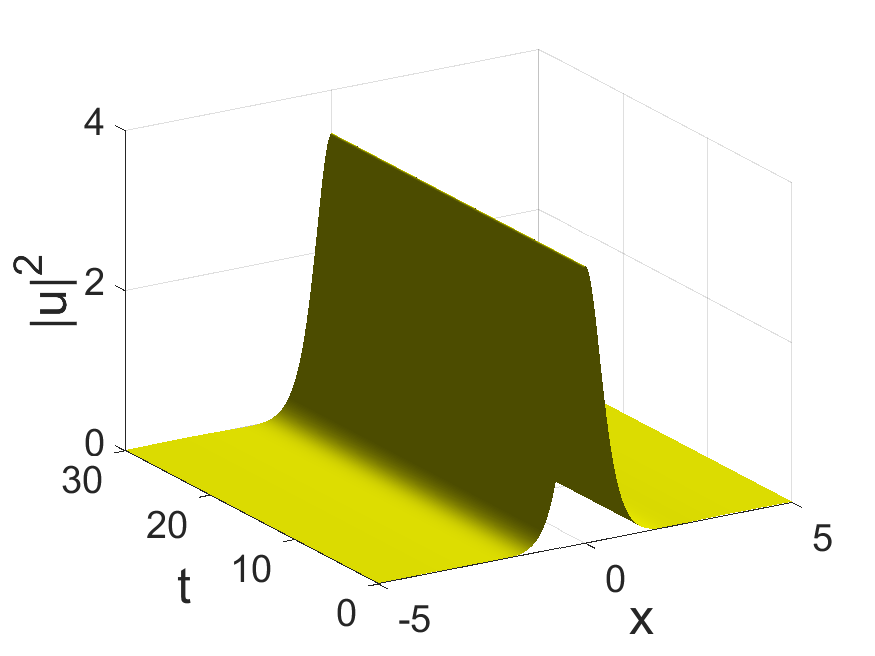}
 \end{minipage}
\hspace{30pt}
\begin{minipage}[t]{0.4\linewidth}
   \includegraphics[width=3in]{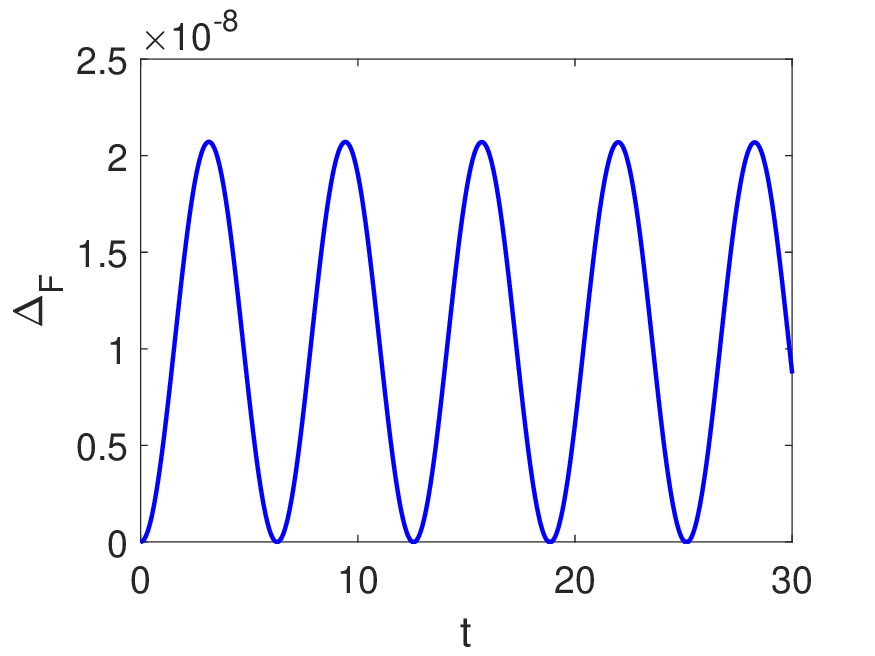}
 \end{minipage}

 \hspace{100pt}
\begin{minipage}[t]{0.4\linewidth}
   \includegraphics[width=3in]{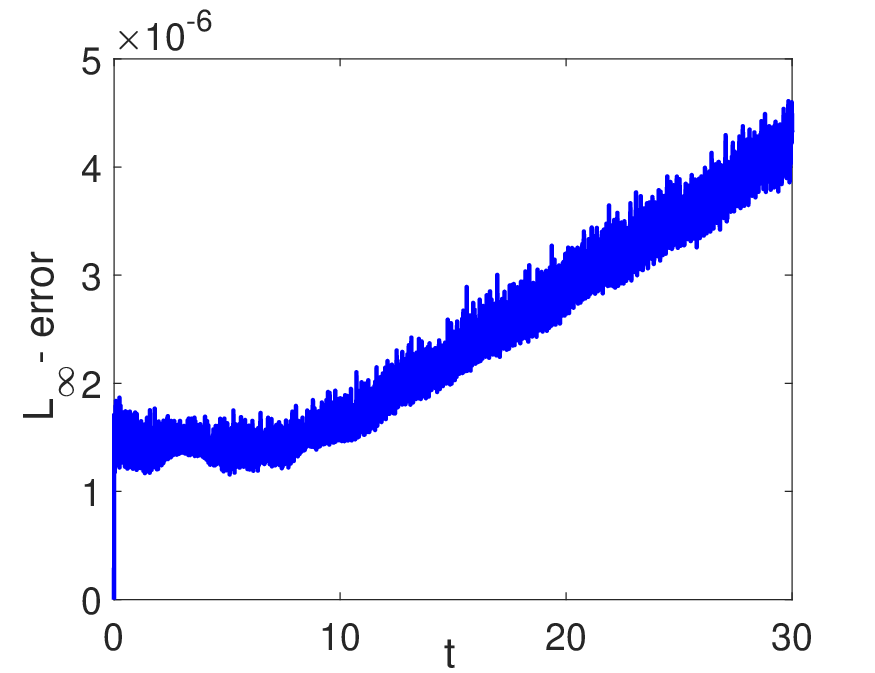}
 \end{minipage}

\caption{{The time evolution of the modulus squared of standing wave solution corresponding to the initial data $u=Q(x)$ with $\beta=0.4$ (top left panel), the variation of the change in the conserved quantity ${ F}$ with time (top right panel) and the variation of the $L_{\infty}$-error with time.}}
 \label{timeevolution-beta04}
\end{figure}

Now, we study the time evolution of the standing wave solution corresponding to the initial data $u=Q(x)$ with $\beta=0.4$ obtained by the Petviashvili's iteration method. The problem is solved in the space interval \mbox{$-1000 \leq x \leq 1000$} up to $T=30$. We take the number of grid points as $N=2^{18}$, $M=30000$. { The top left panel of Figure \ref{timeevolution-beta04} illustrates the time evolution of the modulus squared of the solution to the nonlocal NLS equation with $\beta=0.4$. The quantity
\begin{equation}
    \Delta_F= \abso{ \frac{{ F}(t)-{ F}(0)}{{ F}(0)}}
\end{equation}
indicates the accuracy of the numerical scheme. The top right panel of Figure \ref{timeevolution-beta04}  shows the variation of change in the conserved quantity ${\ F}$ with time. This numerical example shows that the proposed scheme preserves the conserved quantity.  The variation of the $L^{\infty}-$error between the numerical solution and the standing wave solution $u=e^{-it}Q(x)$ using the generated standing wave $Q(x)$ with time  is also illustrated in Figure \ref{timeevolution-beta04}.}
The numerical scheme captures the exact solution for $\beta=0.4$  well enough.

In order to explore the stability of standing wave solution, we first generate the standing wave solution by using Petviashvili's iteration method. Then, we perturb this profile by a factor $r > 0$. Finally, the perturbed profile is used as an initial condition as
\begin{equation*}
    u_0(x)=r Q(x)
\end{equation*}
for the time-splitting method and the evolution of the resulting numerical
approximation is monitored. To check the accuracy of our code, we force the mass conservation error $\Delta_F$ be less than $10^{-4}$ at each time step,
where the mass integral \eqref{mass} is approximated by the trapezoidal rule.  The problems are solved in the space interval $-4000 \leq x \leq 4000$ up to $T=30$ taking the number of grid points as $N=2^{18}, M=7500$ and $\sigma=1$.

\paragraph{Perturbed standing wave solutions in the mass subcritical regime}
We first investigate the time evolution of perturbed standing wave solutions in the mass subcritical case, i.e. $\beta < \frac{1}{2}$.  We choose the initial data \mbox{$u=0.9 Q(x)$} for $\beta=0.4$.  The variation of $\| u\|_{\nx}$  with time and time evolution of the modulus squared of perturbed standing wave solution corresponding to $\beta=0.4$ are  illustrated in Figure \ref{beta0409}.
The  $\nx$-norm of the solution decreases with time for $\beta=0.4$.

\begin{figure}[h!]
 \begin{minipage}[t]{0.45\linewidth}
   \includegraphics[width=3in]{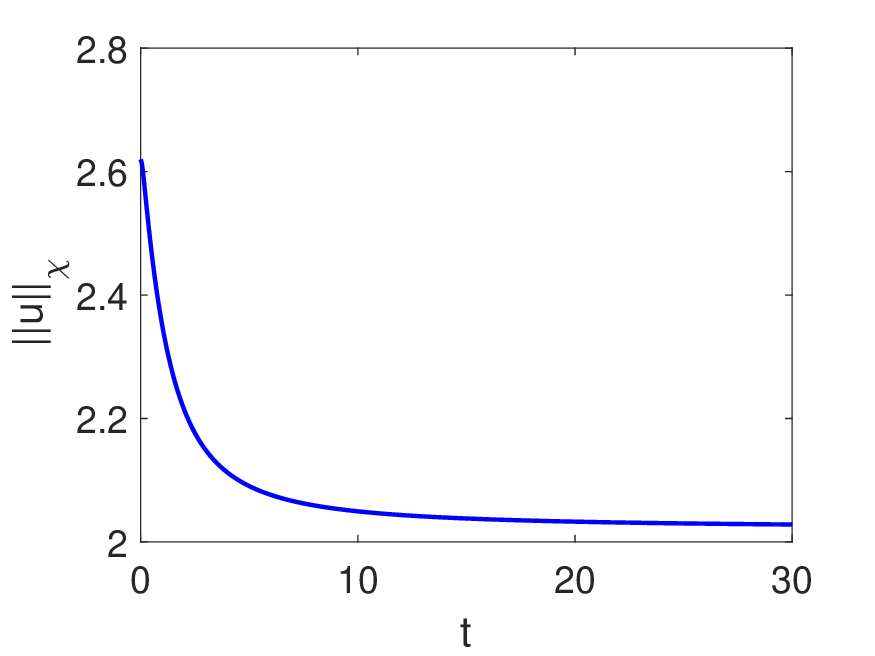}
 \end{minipage}
\hspace{30pt}
\begin{minipage}[t]{0.45\linewidth}
   \includegraphics[width=3in]{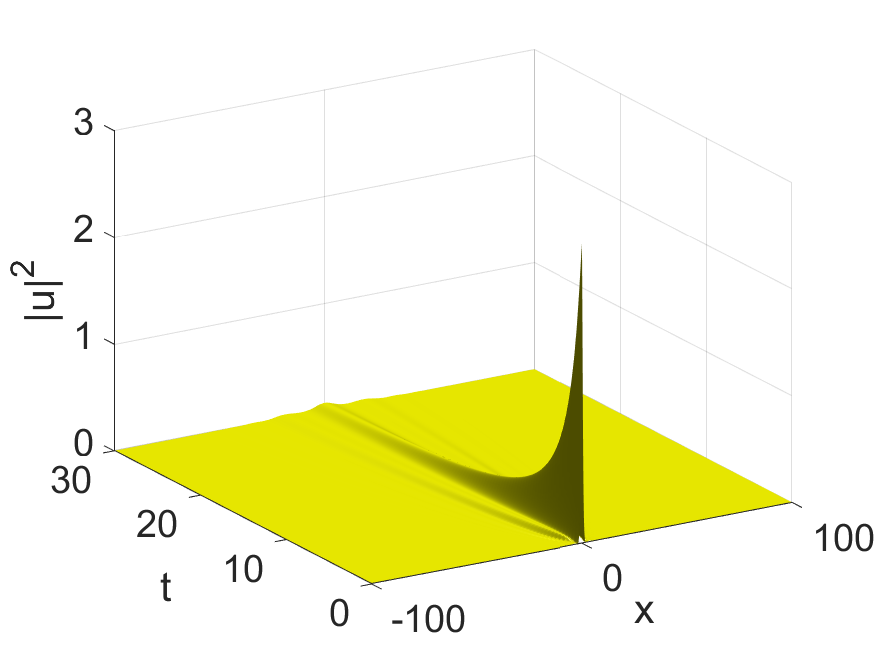}
 \end{minipage}
 \hspace{30pt}
\caption{ The variation of $\| u\|_{\nx}$  with time and time evolution of the modulus squared of perturbed standing wave  solution corresponding to the initial data $u=0.9 Q(x)$ for $\beta=0.4$. } \label{beta0409}
\end{figure}


The variation of $\| u\|_{\nx}$  with time and time evolution of the modulus squared of perturbed standing wave solution corresponding to the initial data $u=1.1 Q(x)$ for $\beta=0.4$ are presented in Figure \ref{beta0411}. We observe that the perturbed standing wave solutions oscillate. The amplitude of oscillations decreases as time increases.  Since the solution is bounded in ${\nx}$, the numerical result indicates that the standing wave is stable for the mass subcritical case. The numerical results are in complete agreement with the
theoretical  predictions given in   Theorem \ref{uniform} and {Theorem \ref{stab-thm}}.

\begin{figure}[h!]
 \begin{minipage}[t]{0.45\linewidth}
   \includegraphics[width=3in]{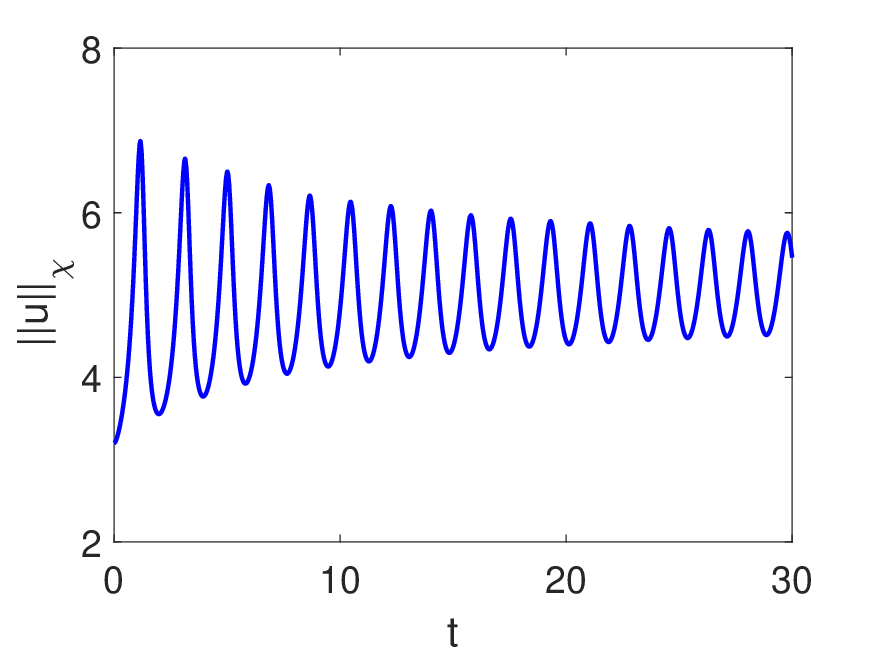}
 \end{minipage}
\hspace{30pt}
\begin{minipage}[t]{0.45\linewidth}
   \includegraphics[width=3in]{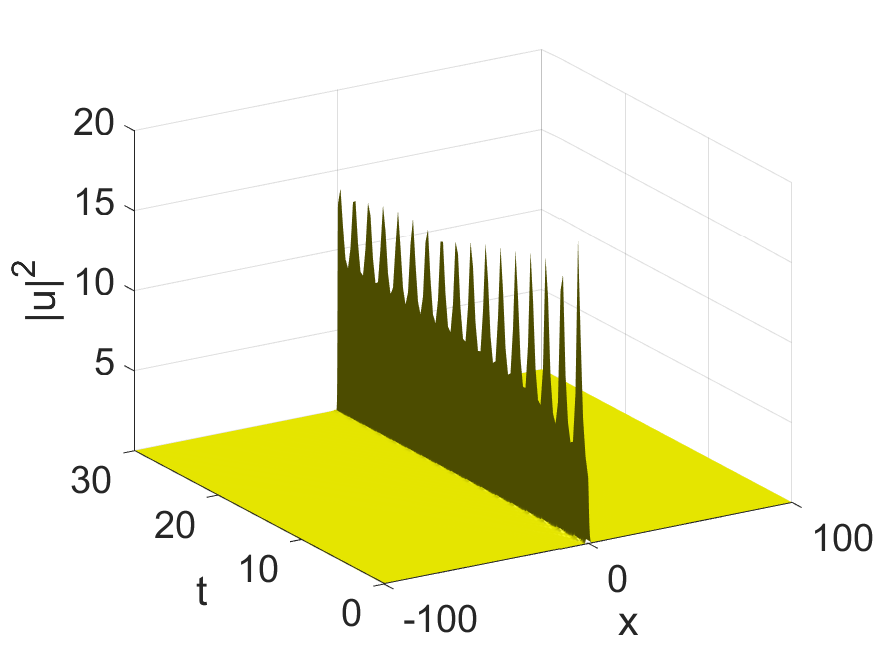}
 \end{minipage}
 \hspace{30pt}
\caption{ The variation of $\| u\|_{\nx}$  with time and time evolution of the modulus squared of perturbed standing wave solution corresponding to the initial data $u=1.1 Q(x)$ for $\beta=0.4$. } \label{beta0411}
\end{figure}


\paragraph{Perturbed standing wave solutions in the mass critical regime}
We begin with initial data having a mass smaller than the standing wave solution, specifically $u=0.9 Q(x)$ for $\beta=0.5$. Figure \ref{beta0509} illustrates the $\nx$ norm of the solution
and time evolution of the modulus squared of perturbed standing wave solution. The amplitude of the solution decreases over time, indicating that the solution remains uniformly bounded in $\nx$. Since $F(u_0)=r^2 F(Q)$ and $\lambda=1$, the condition stated in case (ii) of Theorem \ref{uniform} becomes $r^2<1$. This numerical finding aligns completely with the theoretical result provided in the theorem. In the scenario where $r^2>1$ and $\sigma=1, \beta=0.5$, no theoretical result is available. Both the conditions for Theorem \ref{uniform} and Remark \ref{blow-up-theo} are not met. Figure \ref{beta0511} demonstrates the $\nx$ norm of the solution increasing over time. {The solution seems to exhibit finite-time blow-up. The numerical result indicates that the standing wave is unstable, filling the gap identified in Theorem \ref{stab-thm}.}

\begin{figure}[h!]
 \begin{minipage}[t]{0.45\linewidth}
   \includegraphics[width=3.6in]{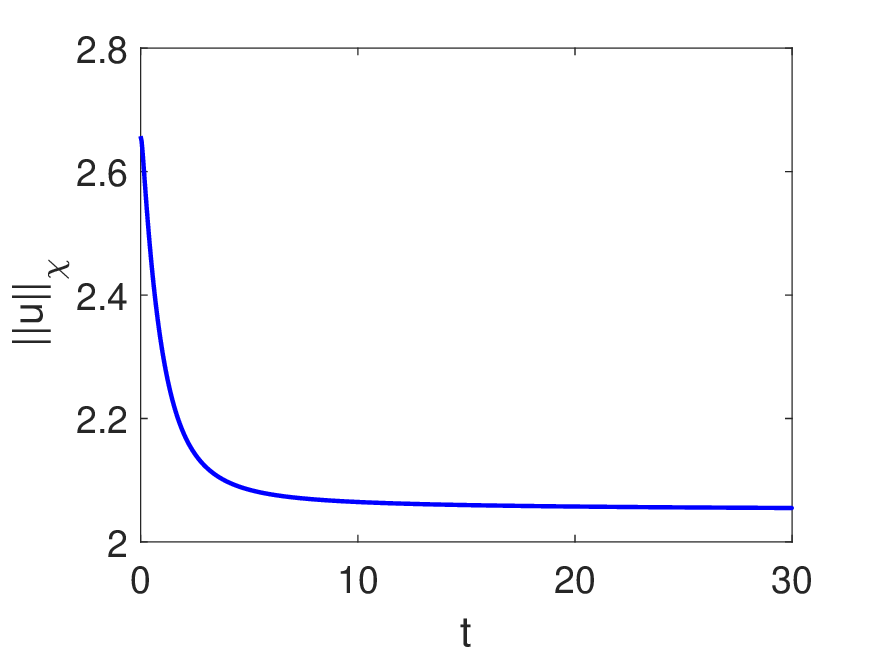}
 \end{minipage}
\hspace{30pt}
\begin{minipage}[t]{0.45\linewidth}
   \includegraphics[width=3.6in]{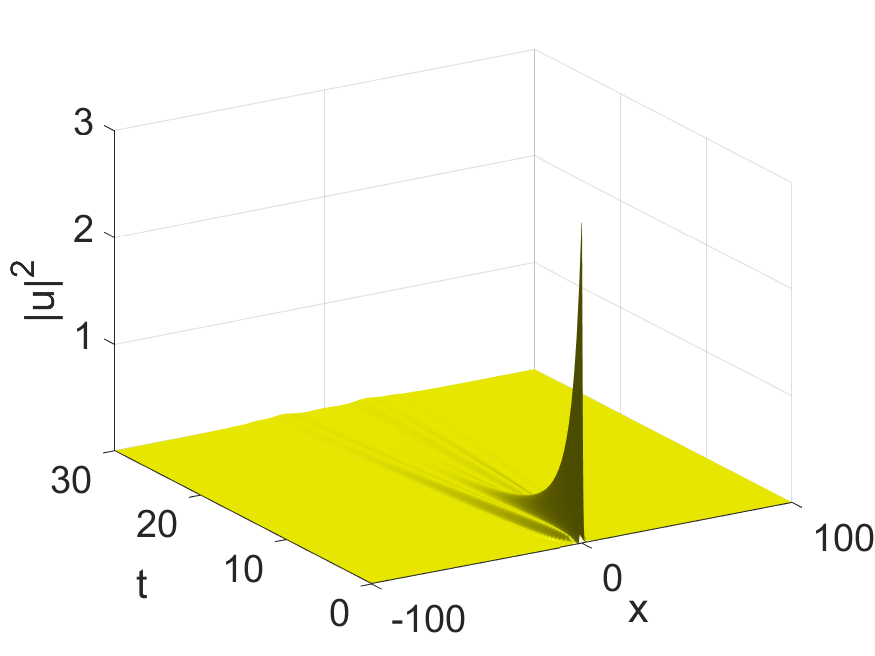}
 \end{minipage}
 \hspace{30pt}
\caption{ The variation of $\| u\|_{\nx}$  with time and time evolution of the modulus squared of perturbed standing wave solution corresponding to the initial data $u=0.9Q(x)$ for $\beta=0.5$. } \label{beta0509}
\end{figure}
\begin{figure}[h!]
 \begin{minipage}[t]{0.45\linewidth}
   \includegraphics[width=3.6in]{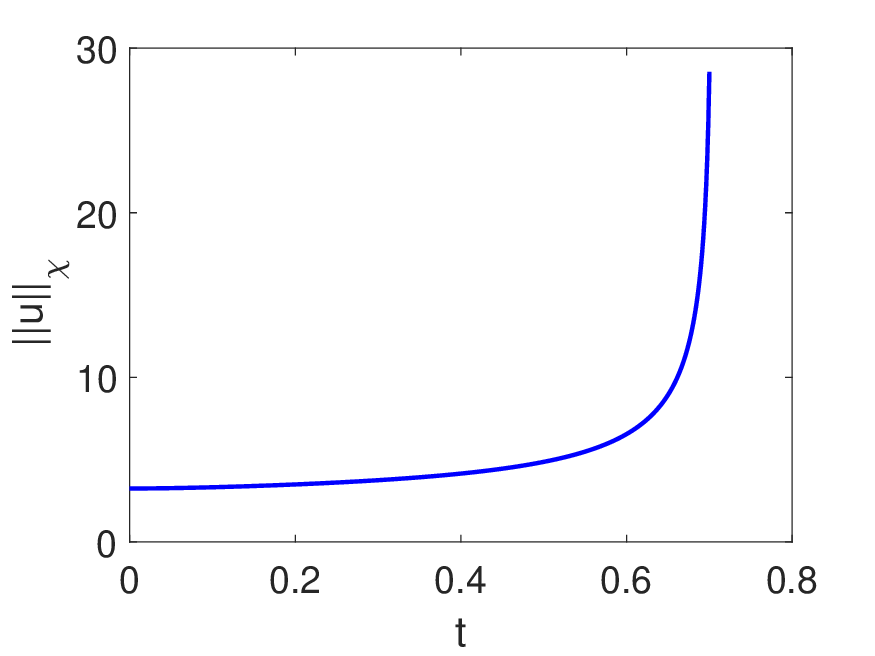}
 \end{minipage}
\hspace{30pt}
\begin{minipage}[t]{0.45\linewidth}
   \includegraphics[width=3.6in]{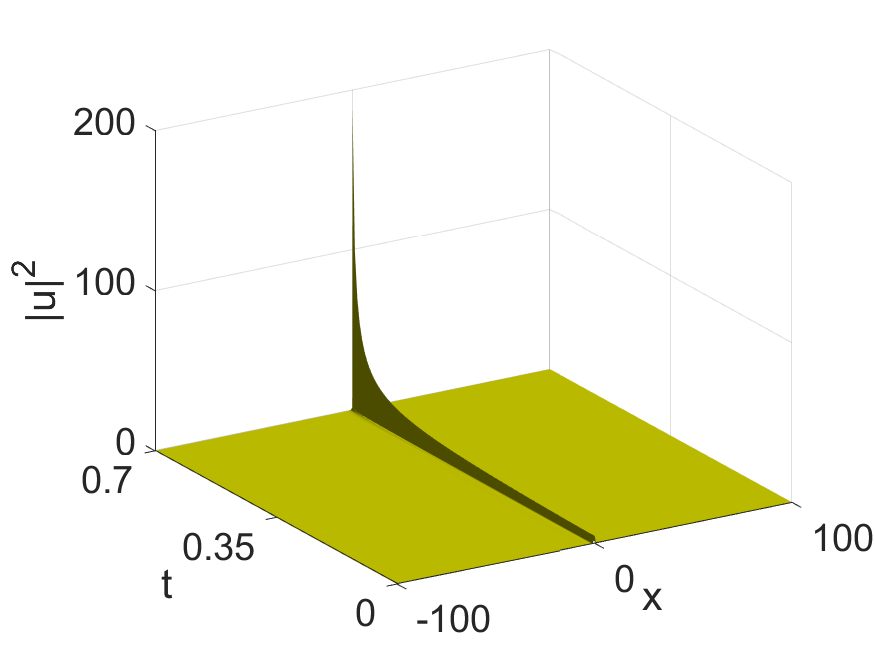}
 \end{minipage}
 \hspace{30pt}
\caption{ The variation of $\| u\|_{\nx}$  with time and time evolution of the modulus squared of perturbed standing wave solution corresponding to the initial data $u=1.1Q(x)$ for $\beta=0.5$. } \label{beta0511}
\end{figure}
\paragraph{Perturbed standing wave solutions in the mass supercritical regime}
We now investigate the time evolution of the perturbed standing wave  solution in the mass supercritical case, where $\beta > \frac{1}{2}$. Figure \ref{beta0809} illustrates the $\nx$ norm of the solution and its temporal evolution with initial data $u=0.9 Q(x)$ for $\beta=0.8$. The amplitude of the solution decreases over time, indicating that the solution remains uniformly bounded in $\nx$.
{Identities \eqref{poho-1}
give
\[
E(Q)=
\frac{\beta(\sigma+1)+\sigma-2}{4(\sigma+1)}
\norm{Q}_{L^{2\sigma+2}}^{2\sigma+2}=
\frac{\beta(\sigma+1)+\sigma-2}{2(\beta(\sigma+1)+\sigma)}
\norm{D^{-\frac\beta2}Q_x}_\lt^2
\]}
and
\begin{equation}
E(rQ)=
 \paar{\frac{\beta(\sigma+1)+\sigma}{2}-r^{2\sigma}}\frac{2r^2}{\beta(\sigma+1)+\sigma-2}E(Q),
\end{equation}
hence part (iii) of Theorem \ref{uniform} yields  $r^{2\sigma}<1$ and
\begin{equation}
    \paar{\frac{\beta(\sigma+1)+\sigma}{2}-r^{2\sigma}}\frac{2r^{ \frac{4\sigma}{\beta(\sigma+1)+\sigma-2}}}{\beta(\sigma+1)+\sigma-2}
<1.  \label{cond2}
\end{equation}
For $r=0.9$, $\beta=0.8$, and $\sigma=1$, the left-hand side of \eqref{cond2} is approximately $0.8$, satisfying the conditions of Theorem \ref{uniform}. The numerical results align well with the analytical findings. The variation of $\| u\|_{\nx}$ over time and the time evolution of the modulus squared of the perturbed standing wave solution corresponding to the initial data $u=1.1Q(x)$ for $\beta=0.8$ are depicted in Figure \ref{beta0811}. Both the $\nx$-norm of the solution and the amplitude of the solution increase over time.  If we take $u_0=rQ$, the conditions of Remark \ref{blow-up-theo} turn into
\[
E(rQ)=\frac{r^2}{2(\sigma+1)}\paar{\frac{\beta(\sigma+1)+\sigma}{2}-r^{2\sigma}}\|Q\|_{2\sigma+2}^{2\sigma+2}<0\iff
r^{2\sigma}
>
\frac{\beta(\sigma+1)+\sigma}2.
\]
If $E(rQ)\geq0$,   conditions of Remark \ref{blow-up-theo} yield  $r^2>1$ and the inequality given by \eqref{cond2}.
For $r=1.1$, $\beta=0.8$ and $\sigma=1$, $r^{2\sigma}< \frac{\beta(\sigma+1)+\sigma}2$. In this case, there is no analytical result. The solution blows up in finite time, and the numerical result indicates that the standing wave  is unstable.

\begin{figure}[h!]
 \begin{minipage}[t]{0.45\linewidth}
   \includegraphics[width=3.6in]{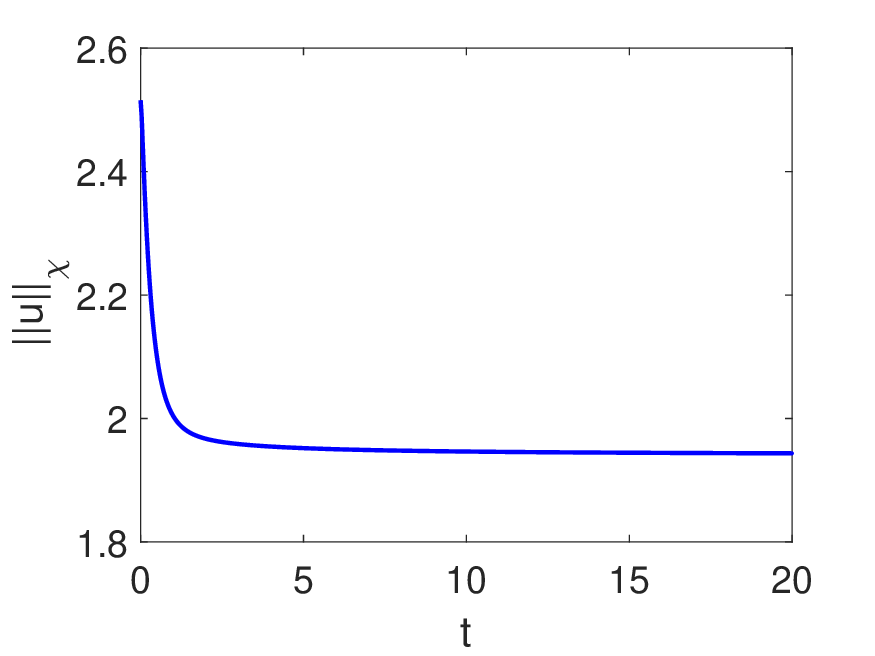}
 \end{minipage}
\hspace{30pt}
\begin{minipage}[t]{0.45\linewidth}
   \includegraphics[width=3.6in]{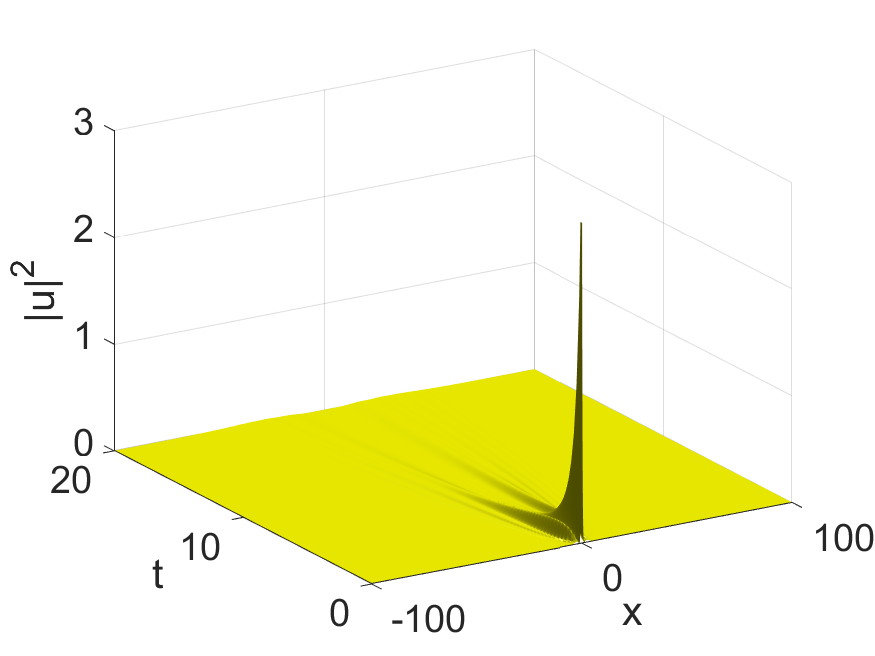}
 \end{minipage}
 \hspace{30pt}
\caption{ The time variation of $\| u(t)\|_{\nx}$  and time evolution of the modulus squared of perturbed standing wave solution corresponding to the initial data $u=0.9Q(x)$ for $\beta=0.8$. } \label{beta0809}
\end{figure}
\begin{figure}[h!]
 \begin{minipage}[t]{0.45\linewidth}
   \includegraphics[width=3.6in]{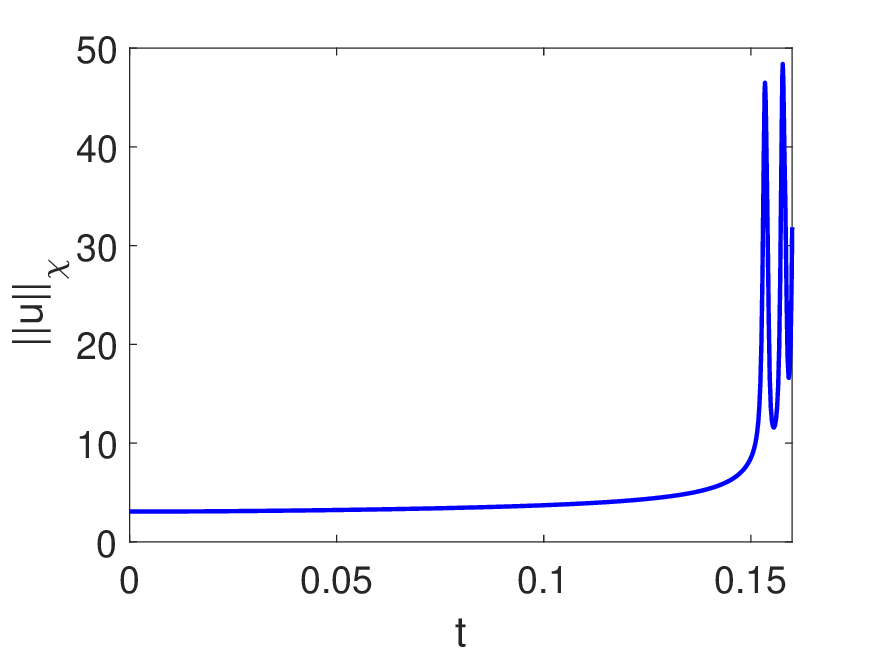}
 \end{minipage}
\hspace{30pt}
\begin{minipage}[t]{0.45\linewidth}
   \includegraphics[width=3.6in]{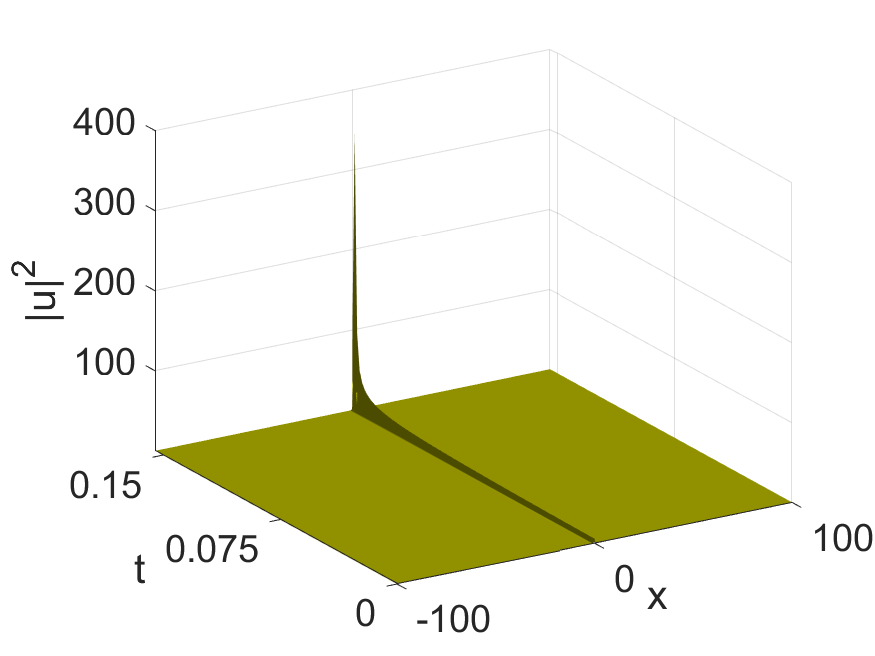}
 \end{minipage}
 \hspace{30pt}
\caption{ The variation of $\| u\|_{\nx}$  with time and time evolution of the modulus squared of perturbed standing wave solution corresponding to the initial data $u=1.1Q(x)$ for $\beta=0.8$. } \label{beta0811}
\end{figure}

\section{Boosted standing waves}\label{Boosted standing waves}

In this section, we study the existence and dynamical properties of boosted standing waves of \eqref{nonlocalNLS}. Such solutions take the form
  $u(x,t)=\ee^{-\ii\omega t}\ff(x-ct)$, where $\omega$ is the wave frequency, and $c$ is the velocity of the wave, and $\ff$ is a complex-valued function. By plugging this expression into \eqref{nonlocalNLS} and replacing $x-ct$ with $x$, we obtain the equation
\begin{equation}\label{stand-nonzero-vel}
  \omega \varphi-\ii c\ff'-\lambda\varphi^{\prime\prime}=\zeta D^{\beta} ( |\varphi|^{2\sigma} \varphi). 
\end{equation}
  In the case $\beta=0$, it is known that \eqref{NLS-0} is Galilei invariant. Therefore,      if $u(x,t)=\ee^{-\ii t}Q(x-ct)$ satisfies \eqref{NLS-0}, then
\begin{equation}
    \ee^{\ii\paar{\frac{c^2}{4\lam}-1}t-\frac{c}{2\lam}\ii x}Q(x-ct)
\end{equation}
also satisfies it, where $Q$ is the solution of \eqref{exact}.

Due to presence of the nonlocal operator $D^\beta$, such a Galilei transformation seems not to exist, so the qualitative properties of boosted standing waves is unpredictable.

It is clear that if $\ff\in\nx$ satisfies \eqref{stand-nonzero-vel}, then the following Pohozaev identity holds:
\begin{equation} \label{pohi}
\omega F(\ff)-\ii c\int_\rr D^{-\frac\beta2}\ff_x \overline{D^{-\frac\beta2}\ff} \dd x+\lam
\norm{D^{1-\frac\beta2}\ff}_\lt^2
=\zeta\|\ff\|_{L^{2\sigma+2}}^{2\sigma+2}.
 \end{equation}
Hence, \eqref{stand-nonzero-vel} has no nontrivial solution when $\zeta=-1$. In the case $\zeta=+1$,   the above identity shows that $c$ and $\omega$ should satisfy $$\frac{c^2}{4\lambda}-\omega<0.$$

Contrary to the case \eqref{stand-3}, the presence of $\ff'$ in \eqref{stand-nonzero-vel} does not allow   to use any scaling, so the arguments used in the proof of Theorem \ref{gs-thm} are not applicable.
So, here we resort to using the following Weinstein-type minimization problem
 \begin{equation}\label{wein}
W_{c}=\inf_{u\in \nx\setminus\{0\}}M_{c}(u)=\inf_{u\in \nx\setminus\{0\}}\frac{\paar{ \norm{D^{-\frac\beta2}u'}_{\dot{\nx}}^2-\ii c \scal{D^{-\frac\beta2}u,D^{-\frac\beta2}u'}+\omega\|u\|_\lt^2}^{\sigma+1}}
{  \|u\|_{L^{2\sigma+2}}^{2\sigma+2}  }.
 \end{equation}
It is clear, from Lemma \ref{gn}, that $ W_c>0$. Moreover, if $u$ is a nontrivial minimizer of \eqref{wein}, then a scaled function of $u$ satisfies \eqref{stand-nonzero-vel}.
\begin{theorem}
     Let $\lambda,\omega,\sigma>0$ and $-1<\beta<2$.  Assume that
\[
\max\sett{0,\frac{-\beta}{1+\beta} }< \sigma<  \begin{cases}
     \frac{2-\beta }{\beta-1},&\beta>1,\\
    \infty,&\beta\leq1,
\end{cases}
\]
and $\omega>\frac{c^2}{4\lambda}$, then there exists a minimizer $\ff\in \nx\setminus\{0\}$ of \eqref{wein} which satisfies (after a scaling) \eqref{stand-nonzero-vel}.
\end{theorem}
The proof of this theorem is aligned with the lines of one of Theorem 2.1 in \cite{bfv}, so we omit the details.

\subsection{Numerical generation of boosted standing waves}
Since the exact solutions for boosted standing waves are unknown for $\beta \neq 0$, we generate a two-parameter family of solutions numerically.  Applying the Fourier transform
to equation \eqref{stand-nonzero-vel} gives
\begin{equation}
  (\omega+ck+\lambda k^2) \widehat{{\varphi}}=\zeta |k|^{\beta} \widehat{ (|\varphi|^{2\sigma} \varphi)}.  \label{standingwave-nonzero}
\end{equation}
 The Petviashvili method  for equation \eqref{standingwave-nonzero} is given by
\begin{equation}
\widehat{\varphi}_{n+1}(k) =\zeta( M_n)^{\nu}  \frac{|k|^{\beta} \widehat{ (|\varphi|^{2\sigma} \varphi})}{\omega+c k+\lambda k^2}
\label{petviashvilischemenonzero}
\end{equation}
with
\begin{equation*}
  M_n=\frac{\int_{\mathbb{R}} [(\omega+ck+\lambda k^2) [\widehat{{\varphi}}_{n}(k)]^2 dk }
  {\zeta \int_{\mathbb{R}}  { |k|^{\beta} \widehat{ (|\varphi|^{2\sigma} \varphi)} \widehat{{\varphi}}_{n}(k) dk }}.
\end{equation*}

\begin{figure}[h!]
 \begin{minipage}[t]{0.45\linewidth}
   \includegraphics[width=3in]{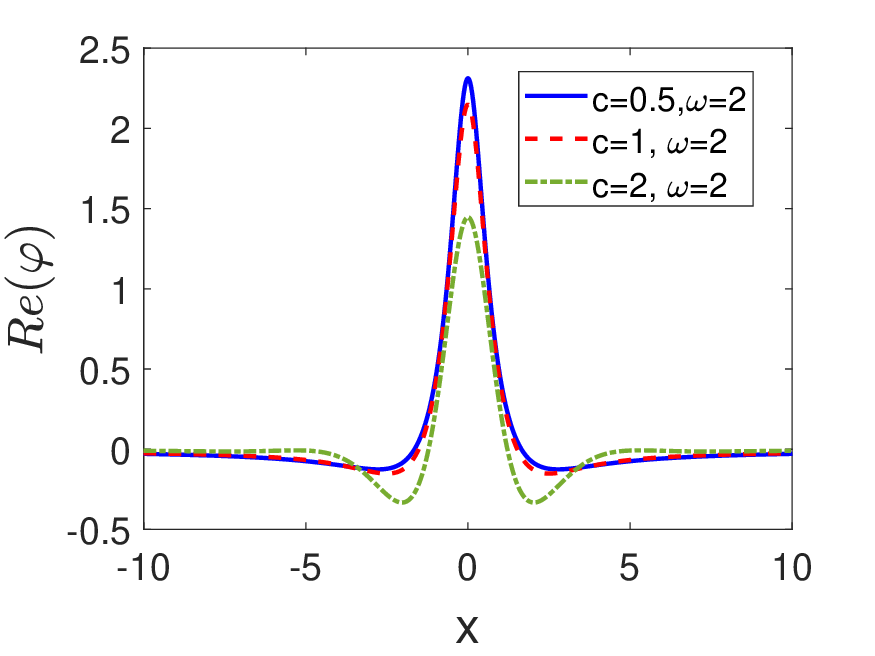}
 \end{minipage}
\hspace{30pt}
\begin{minipage}[t]{0.45\linewidth}
   \includegraphics[width=3in]{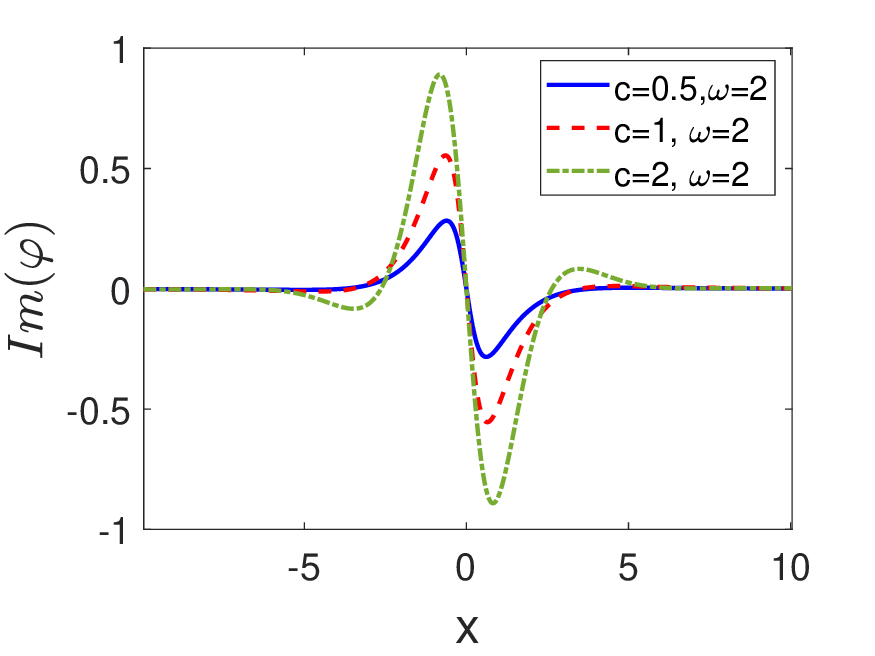}
 \end{minipage}
 \caption{ The real and imaginary part of the profiles of the solutions for a fixed value of $\omega=2$ and various values of $c=0.5,1,2$ with $\beta=0.3, \sigma=1$.} \label{boosted_sw_wfixed}
 \end{figure}
 \begin{figure}[h!]
 \begin{minipage}[t]{0.45\linewidth}
   \includegraphics[width=3in]{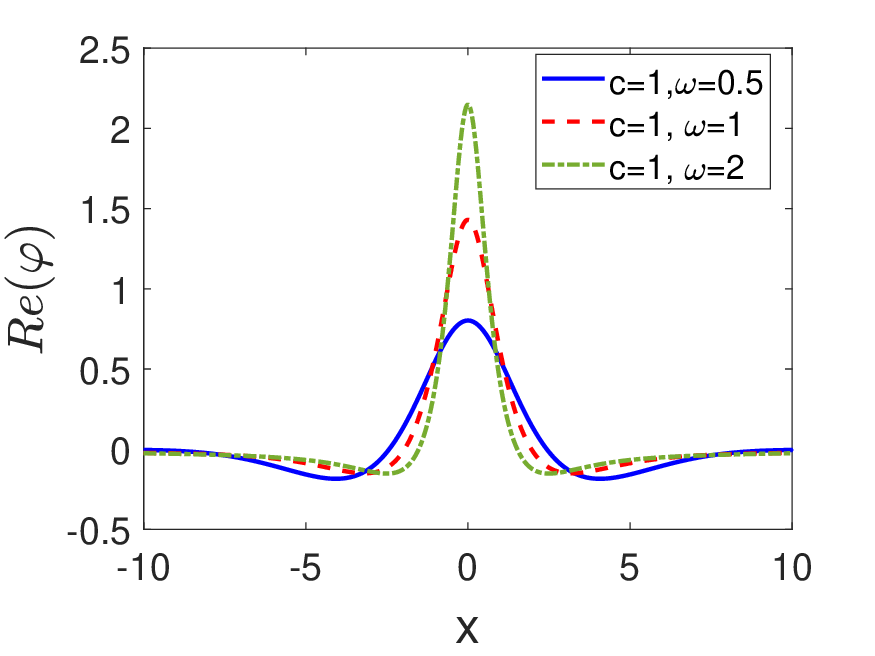}
 \end{minipage}
 \hspace{30pt}
\begin{minipage}[t]{0.45\linewidth}
   \includegraphics[width=3in]{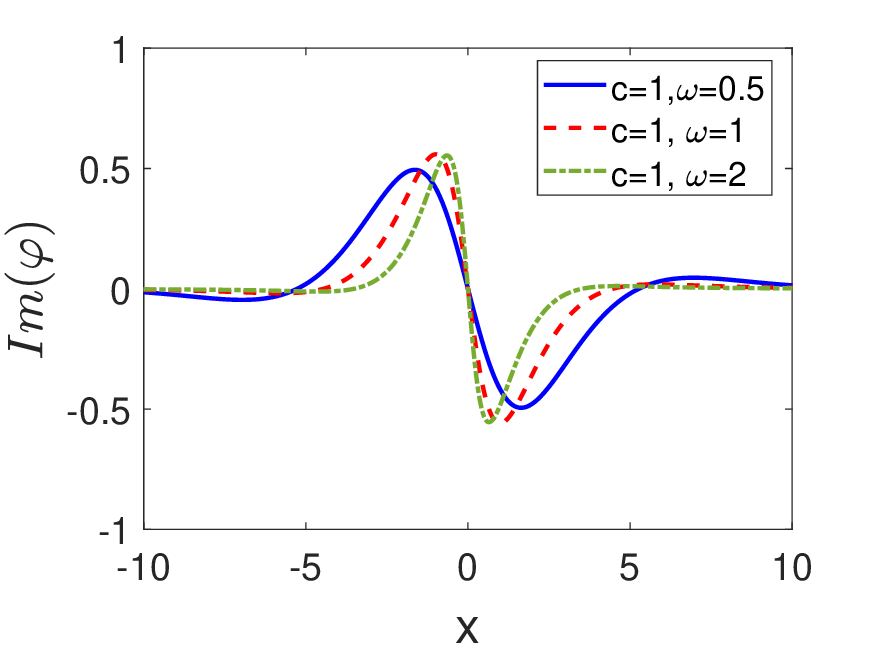}
 \end{minipage}
\caption{The real and imaginary part of the profiles of the solutions for a fixed value of $c=1$ and various values of $\omega=0.5,1,2$ with $\beta=0.3, \sigma=1$.} \label{boosted_sw_cfixed}
\end{figure}
Now, we  generate boosted standing waves numerically for several values of $c$ and $w$.  Here, we  are interested in the focusing case ($\lambda=1, \zeta=1$) choosing
  the space interval as  \mbox{$x \in [-2^{15},2^{15}] $} and taking the number of grid points as $N=2^{21}$. Figure \ref{boosted_sw_wfixed} shows the real and imaginary part of the profiles of the solutions for a fixed value of $\omega$ and various values of $c$ for $\beta=0.3, \sigma=1$. The amplitude of the real part of the solution decreases when $c$ increases. The amplitude of imaginary part of the solution increases when $c$ increases. The real and imaginary part of profiles for a fixed value of $c$ and several values of $\omega$ for $\beta=0.3, \sigma=1$  are illustrated in Figure \ref{boosted_sw_cfixed}. As it is seen from the figure, the amplitude of the real part of the solution increases when $\omega$ increases. {The amplitude of imaginary part of the solution increases when $\omega$ increases.}

\subsection{Numerical study of stability of boosted standing waves}

In this subsection, we  study numerically the stability of boosted standing wave $\ff=\ff_{c,\omega}$ for various $(c,\omega)$. Similar to subsection \ref{stab-sub}, the stability of such waves can be analyzed based on the convexity of the Lyapunov function.
\[
d(\omega)=d_c(\omega)=E(\ff)+\frac{\omega}{2}F(\ff)+\frac{c}{2}P(\ff).
\]
Observe from \eqref{pohi} that
\begin{equation}
\displaystyle d(\omega)=\frac{\sigma}{2(\sigma+1)}\int_ {\mathbb{R}} |\varphi(x)|^{2\sigma+2} \dd x.
\end{equation}
For a fixed speed $c$, if we assume that the curve $\omega \mapsto \ff_{c,\omega}$ is $C^2$,the convexity of $d$ is connected to the sign of $d''$, where $'=\frac{\dd}{\dd\omega}$.

Note that the boosted standing waves exist whenever $\omega > \frac{c^2}{4\lambda}$. { In order to
study convexity of $d$, we first discretize the interval $(\frac{c^2}{4\lambda}, 3)$. $50$ grid points are located in the interval $\omega \in (\frac{c^2}{4\lambda}, 3)$. For each $\omega $  value and a fixed $c$,  the boosted standing wave $\ff_{c,\omega}$ is generated by using the algorithm (3.6). Then, the approximate value of $d$ in (3.7) is approximated by the trapezoidal rule. Finally,
the second-order derivative $d''$  is calculated by a central difference approximation. }

\begin{figure}[h!]
 \begin{minipage}[t]{0.45\linewidth}
   \includegraphics[width=3in]{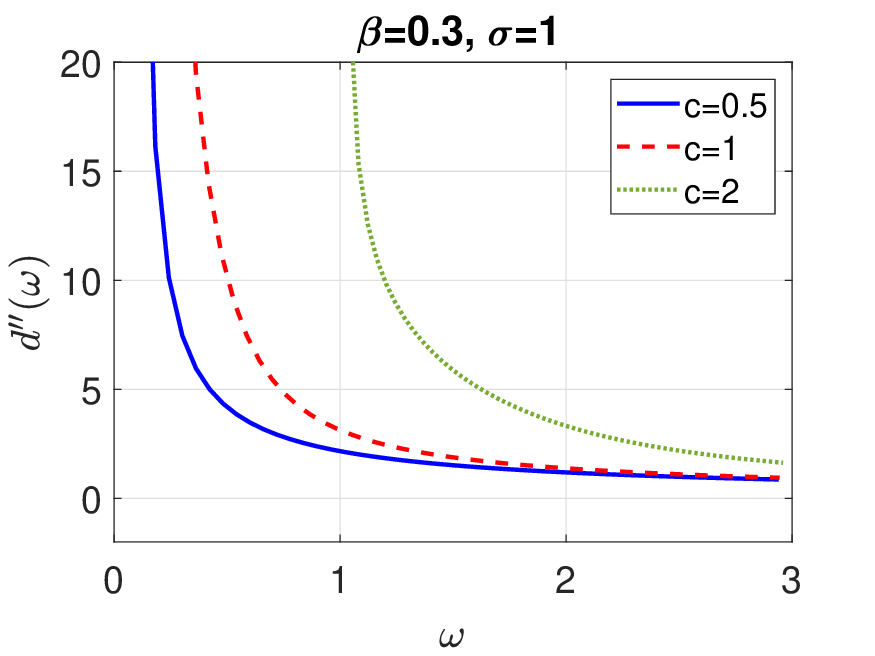}
 \end{minipage}
 \hspace{30pt}
\begin{minipage}[t]{0.45\linewidth}
   \includegraphics[width=3in]{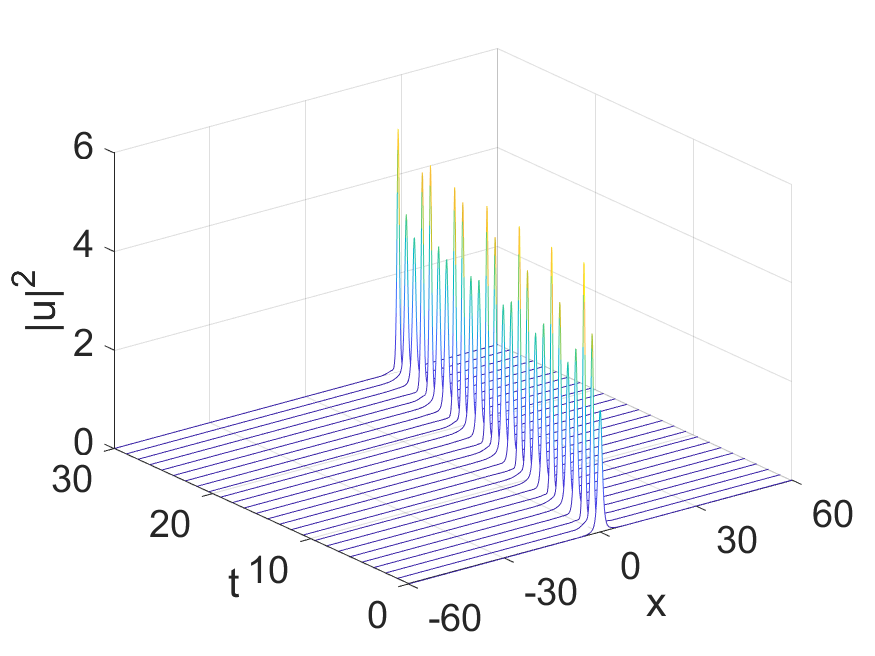}
 \end{minipage}
\caption{Plots of $d^{\prime\prime}(\omega)$ for $\beta=0.3, \sigma=1$ with $c=0.5,1,2$ and time evolution of the modulus squared of perturbed boosted standing wave
solution corresponding to the initial data $u = 1.1\varphi(x)$ for $ c=1, \omega=1$} \label{beta03sigma1}
\end{figure}

\begin{figure}[h!]
 \begin{minipage}[t]{0.45\linewidth}
 \includegraphics[width=3in]{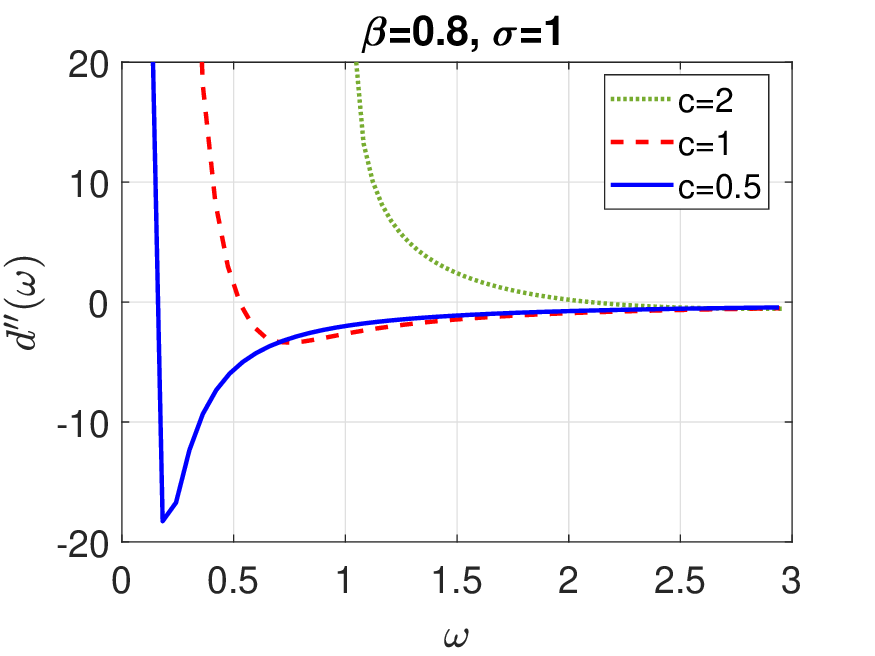}
 \end{minipage}
  \begin{minipage}[t]{0.45\linewidth}
\includegraphics[width=3in]{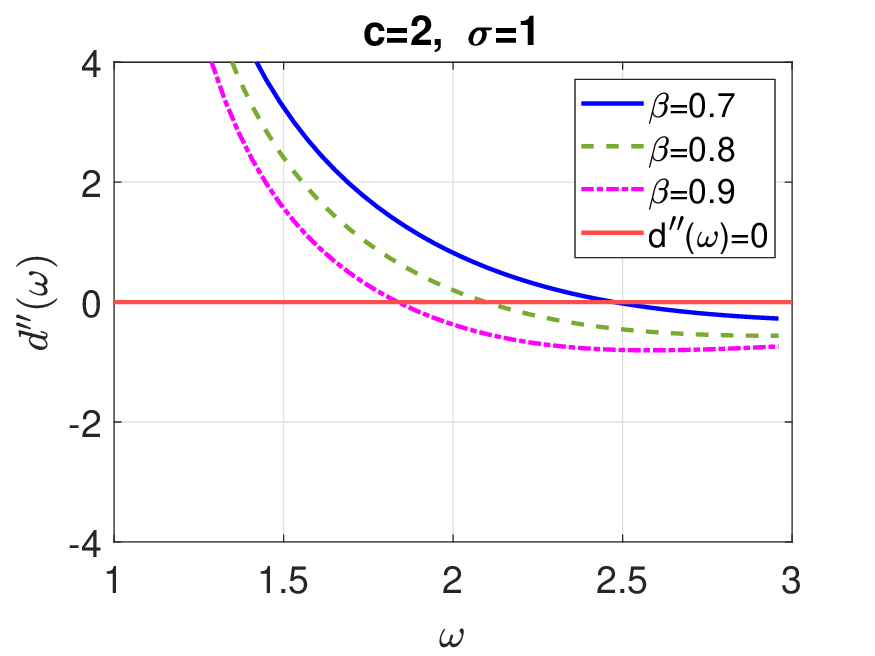}
\end{minipage}
 \caption{ Plots of $d^{\prime\prime}(\omega)$ for $\beta=0.8, \sigma=1$ with different speeds and  plots of $d^{\prime\prime}(\omega)$ for $c=2, \sigma=1$ with different $\beta$ values. } \label{sigma1}

 \end{figure}

 \begin{figure} [h!]
\begin{minipage}[t]{0.45\linewidth}
   \includegraphics[width=3in]{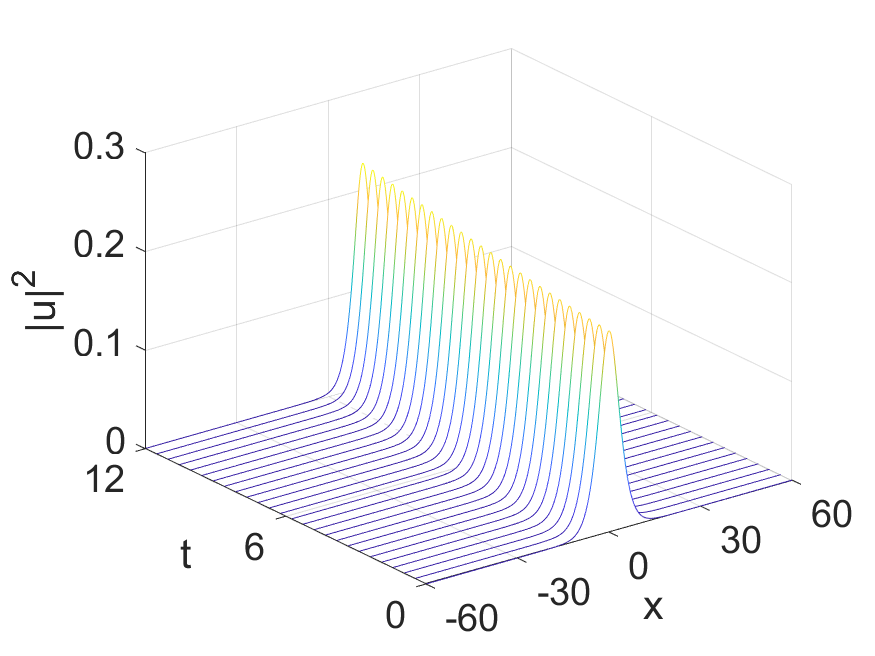}
 \end{minipage}
\hspace{30pt}
 \begin{minipage}[t]{0.45\linewidth}
   \includegraphics[width=3in]{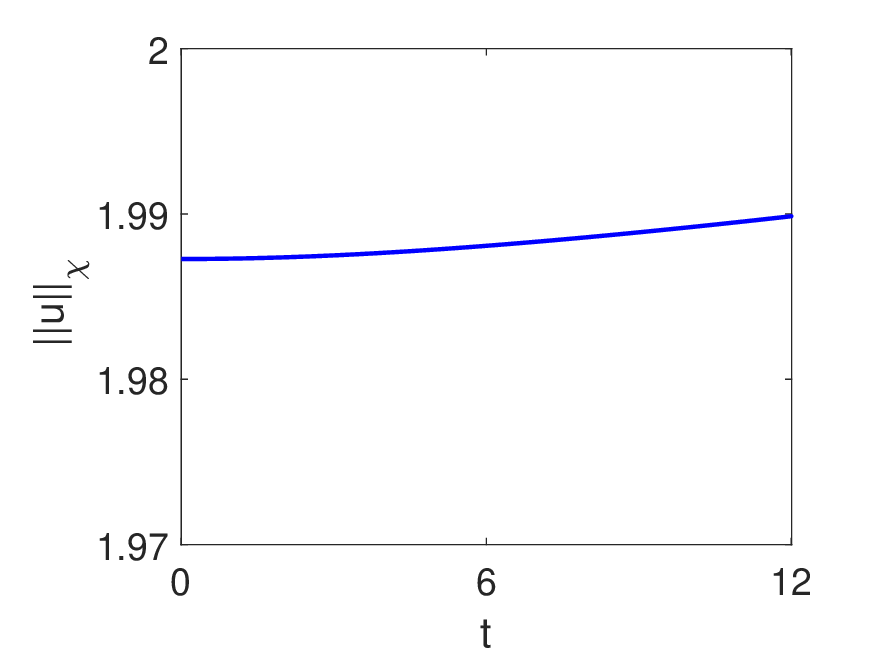}
 \end{minipage}
\caption{ Time evolution of  the modulus squared of perturbed boosted standing wave
solution corresponding to the initial data $u = 1.1\varphi(x)$ for $  c=1, \omega=0.4, \beta=0.8, \sigma=1$ and the variation of  $\| u(t)\|_{\nx}$   with time.  }    \label{beta08sigma1c1w04}
\end{figure}

\begin{figure} [h!]
\begin{minipage}[t]{0.45\linewidth}
   \includegraphics[width=3in]{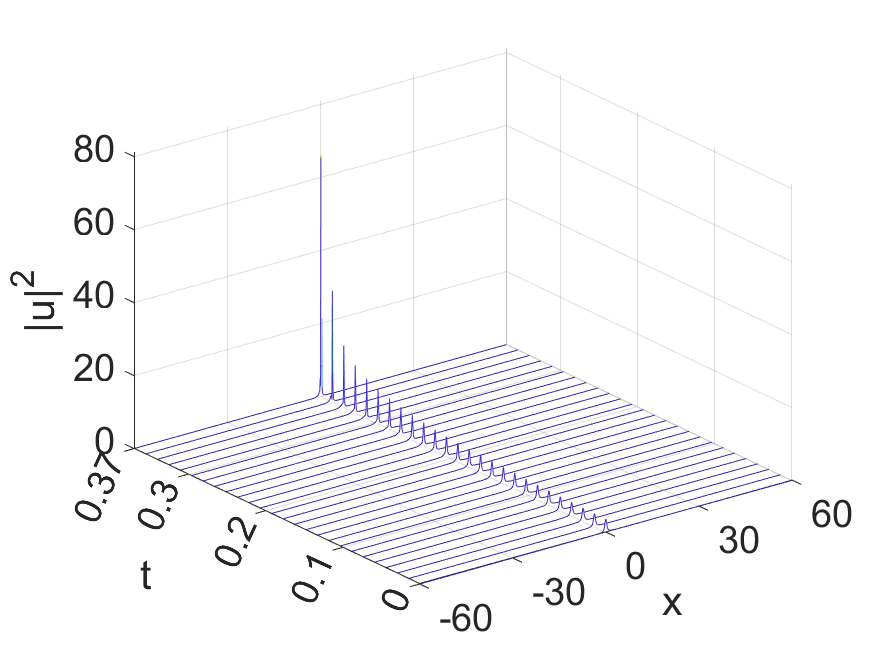}
 \end{minipage}
\hspace{30pt}
 \begin{minipage}[t]{0.45\linewidth}
   \includegraphics[width=3in]{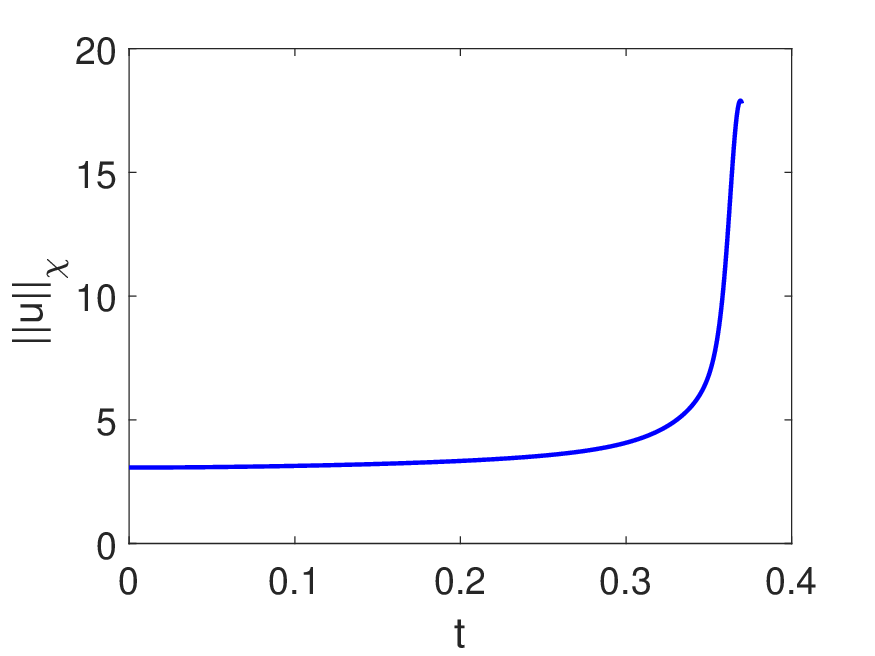}
 \end{minipage}
\caption{ Time evolution of  the modulus squared of perturbed boosted standing wave
solution corresponding to the initial data $u = 1.1\varphi(x)$ for  $ c=1, \omega=1,\beta=0.8, \sigma=1 $ and the variation of  $\| u(t)\|_{\nx}$   with time. } \label{beta08sigma1c1w1}
\end{figure}

\noindent
First, we present the results for the cubic focusing nonlocal NLS equation, setting $\zeta=1, \lambda=1$. The left panel of Figure \ref{beta03sigma1} shows the plots of $d^{\prime\prime}(\omega)$ for $\beta=0.3, \sigma=1$ for various values of $c$.   Figure \ref{beta03sigma1} illustrates that we have stability for all $c$ (within the range of computation performed) {since the sign of $d''(\omega)>0$. }Time evolution of the modulus squared of perturbed boosted standing wave
solution corresponding to the initial data $u = 1.1\varphi(x)$ for $\beta=0.3, c=1, \omega=1$ is depicted in the right panel of Figure  \ref{beta03sigma1}. This figure shows that the boosted standing wave moves oscillatingly to the right with speed $1$ for a long time. Next, we choose $\beta=0.8, \sigma=1$ satisfying
$\sigma>\frac{2-\beta}{1+\beta}$. The left panel of Figure \ref{sigma1} illustrates that {$d''(\omega)$ changes sign. Therefore,
there exists a critical wave frequency $\omega_c$  such that the wave is stable when $\omega < \omega_c$ and unstable $\omega > \omega_c$. }  It can be  seen that the critical $w_c$  increases as the speed increases for a fixed $\sigma$ and $\beta$.  The right panel of Figure \ref{sigma1}
shows the plots of $d^{\prime\prime}(\omega)$ for $c=2, \sigma=1$ for various values of $\beta$. We observe that the critical wave frequency $w_c$ is a decreasing function of $\beta$ for a fixed $c$ and $\sigma$.
Time evolution of  the modulus squared of perturbed boosted standing wave
solution corresponding to the initial data $u = 1.1\varphi(x)$ for $ (i)~ c=1, \omega=0.4$, $ (ii) ~c=1, \omega=1$ and the variations of  $\| u(t)\|_{\nx}$   with time
are presented in Figures \ref{beta08sigma1c1w04} and \ref{beta08sigma1c1w1}, respectively.  Figure \ref{beta08sigma1c1w04} shows that the boosted standing wave
retains its shape and $\nx$ norm of the solution becomes bounded. However,  the amplitude of the boosted solitary wave  and $\nx$ norm of the solution increase very rapidly near $t=0.37$ in Figure \ref{beta08sigma1c1w1}. The numerical result indicates that the boosted solitary wave  is  orbitally unstable.

\noindent
Now, we investigate the focusing nonlocal NLS equation, setting $\zeta=1, \lambda=1$, with $\sigma=1.5$.   We choose $\beta=0.1, \sigma=1.5$ satisfying
$\sigma <\frac{2-\beta}{1+\beta}$. Figure \ref{beta01sigma15} illustrates that we have stability for all $c$ (within the range of computation performed). Time evolution of  the modulus squared of perturbed boosted standing wave
solution corresponding to the initial data $u = 1.1\varphi(x)$ for $  c=1, \omega=1$ and the variation of  $\| u(t)\|_{\nx}$   with time
are presented in Figure \ref{beta1over10sigma1.5}. The solution moves oscillatingly to the right with speed $1$. Since the ${\nx}$ norm of the solution is bounded, the numerical results indicate that the boosted solitary wave is stable in this case.
Next, we choose $\beta=0.5, \sigma=1.5$ satisfying
$\sigma>\frac{2-\beta}{1+\beta}$. The left panel of Figure \ref{beta05sigma1.5} illustrates that there exists a critical $\omega_c$ value such that the wave  is stable  when $\omega < \omega_c$ and unstable $\omega > \omega_c$, except $c=0.5$.
Time evolution of  the modulus squared of perturbed boosted standing wave
solution corresponding to the initial data $u = 1.1\varphi(x)$ for $ c=1, \omega=1$
is presented in the right panel of  Figure \ref{beta05sigma1.5}. The amplitude of the boosted solitary wave  increases very rapidly near $t=0.3$.   The wave is orbitally unstable since $d^{\prime\prime}(1)<0$, as we expected from the left panel of Figure \ref{beta05sigma1.5}.

\begin{figure}[h!]
\begin{center}
   \includegraphics[width=3in]{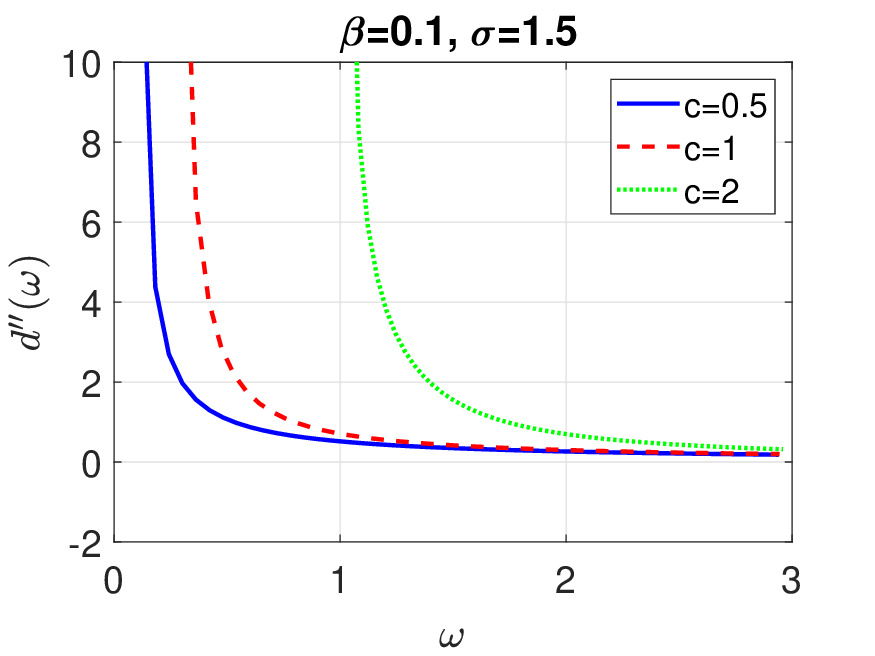}
 \caption{Plots of $d^{\prime\prime}(\omega)$ for $\beta=0.1, \sigma=1.5$ with $c=0.5, 1, 2$. }  \label{beta01sigma15}
\end{center}
\end{figure}

\begin{figure}[h!]
\begin{minipage}[t]{0.45\linewidth}
   \includegraphics[width=3in]{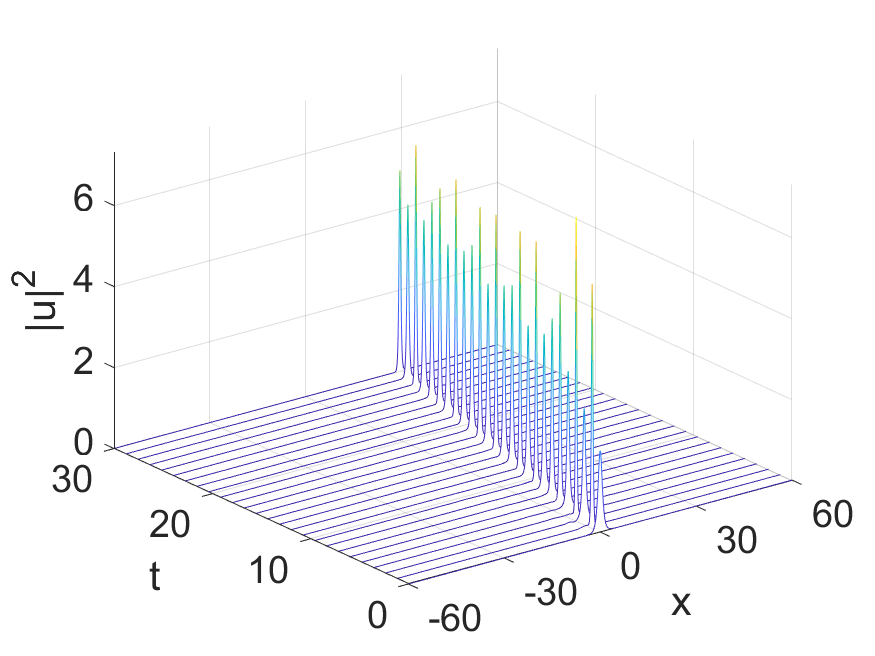}
 \end{minipage}
\begin{minipage}[t]{0.45\linewidth}
   \includegraphics[width=3in]{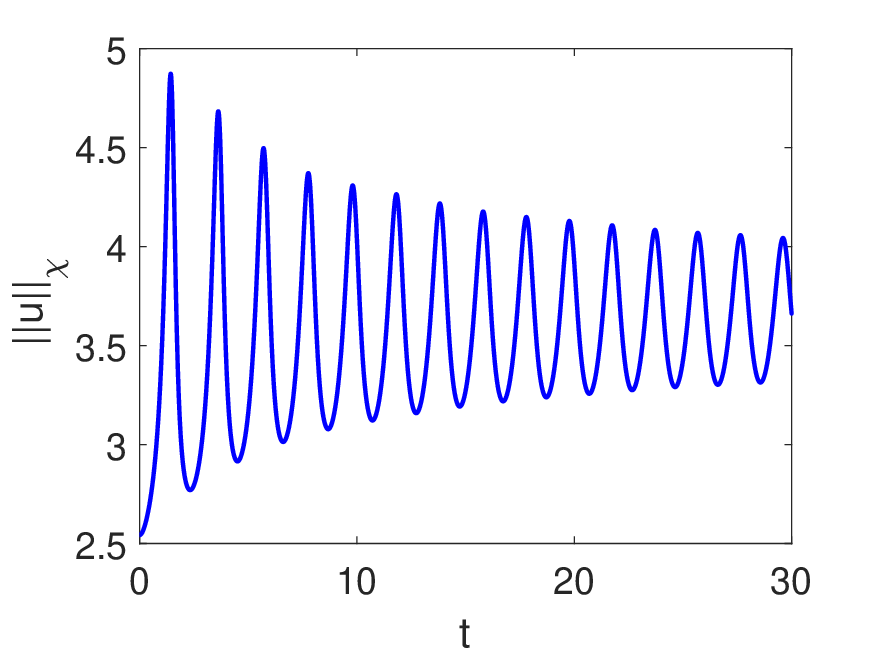}
 \end{minipage}
 \caption{ Time evolution of  the modulus squared of perturbed boosted standing wave
solution corresponding to the initial data $u = 1.1\varphi(x)$ for  $ c=1, \omega=1,\beta=0.1, \sigma=1.5 $ and the variation of  $\| u(t)\|_{\nx}$   with time. } \label{beta1over10sigma1.5}
\end{figure}

\begin{figure}[h!]
 \begin{minipage}[t]{0.45\linewidth}
   \includegraphics[width=3in]{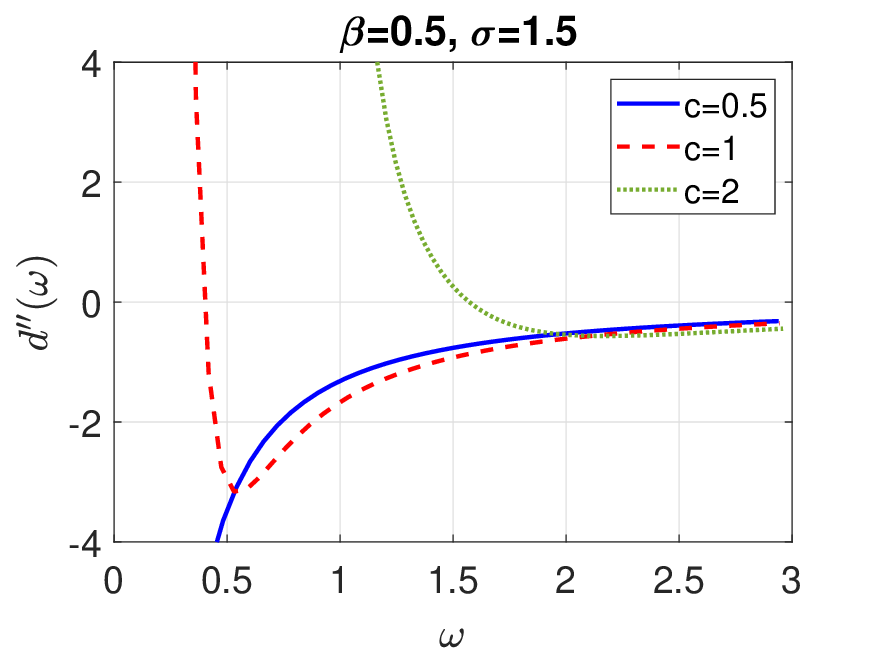}
 \end{minipage}
 \hspace{30pt}
\begin{minipage}[t]{0.45\linewidth}
   \includegraphics[width=3in]{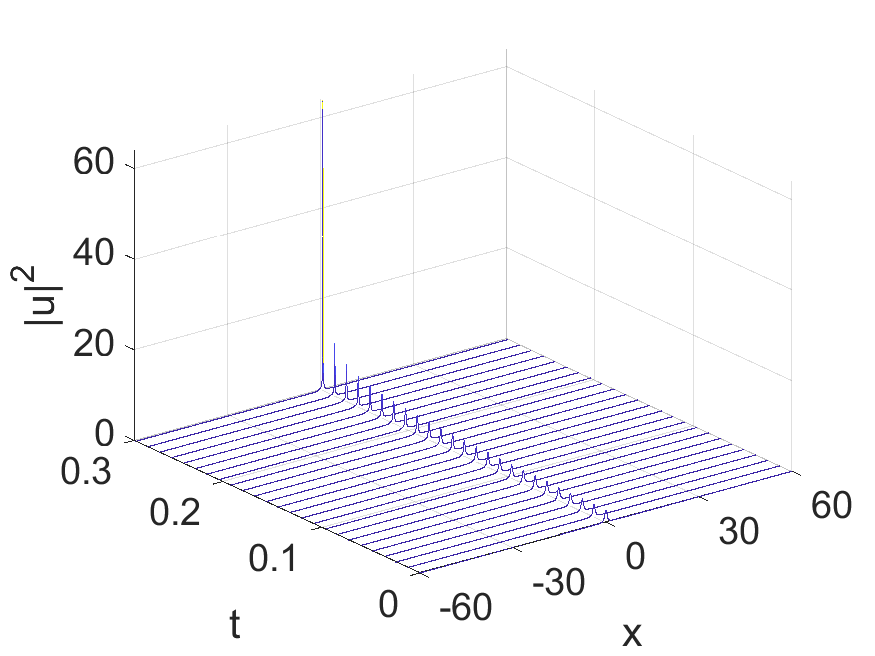}
 \end{minipage}
\caption{Plots of $d^{\prime\prime}(\omega)$ for $\beta=0.5, \sigma=1.5$ with $c=0.5, 1, 2$ and  time evolution of  the modulus squared of perturbed boosted standing wave
solution corresponding to the initial data $u = 1.1\varphi(x)$ for  $ c=1, \omega=1,$ \mbox{$\beta=0.5, \sigma=1.5.$}} \label{beta05sigma1.5}
\end{figure}

\section{The behavior of solutions of the nonlocal NLS equation in the semi-classical limit}\label{semi-classical limit}

In this section, we consider the semi-classical limit of the nonlocal NLS equation  in both
the focusing and the defocusing cases. Let $\epsilon$ be a small parameter $(0 < \epsilon \ll 1)$.
We introduce the slowly varying variables  \mbox{$({x},{t})= (\epsilon \tilde{x}, \epsilon \tilde{t}\,)$} and define
\begin{equation}
     u(\epsilon {x}, \epsilon{t}\,)= u^{\epsilon}(\tilde{x},\tilde{t}).
\end{equation}
We drop the tildes for the sake of convenience so that the scaled nonlocal NLS equation
\begin{equation}
     \ii  u^{\epsilon}_{{t}}-\lambda \epsilon u^{\epsilon}_{{x}{x}}=\zeta {\epsilon}^{\beta-1} D^{\beta} (| u^{\epsilon} |^{2\sigma} u^{\epsilon}). \label{nonlocalNLSepsilon}
\end{equation}

\begin{figure}[h!]
 \begin{minipage}[t]{0.45\linewidth}
   \includegraphics[width=3in]{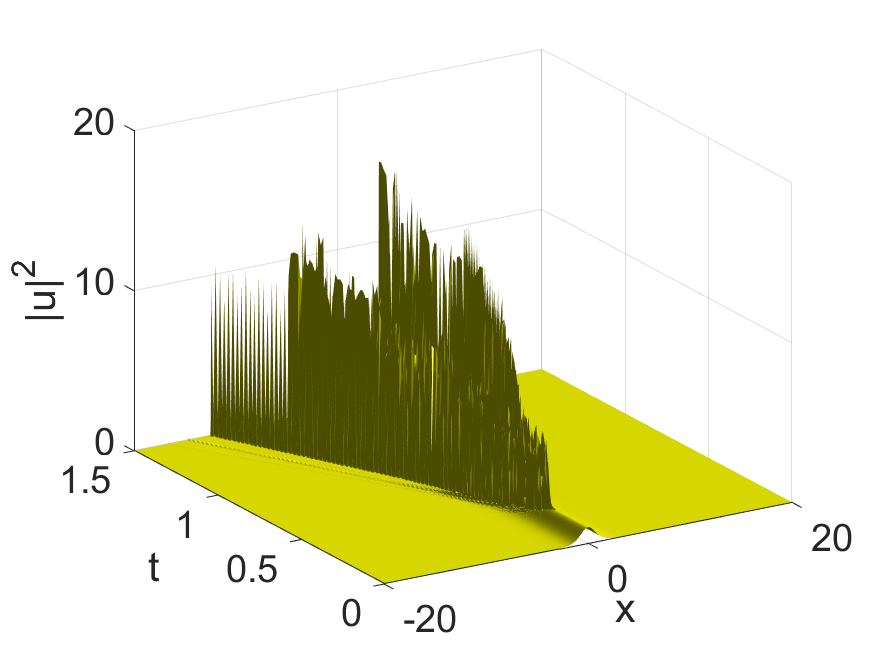}
 \end{minipage}
\hspace{30pt}
\begin{minipage}[t]{0.45\linewidth}
   \includegraphics[width=3in]{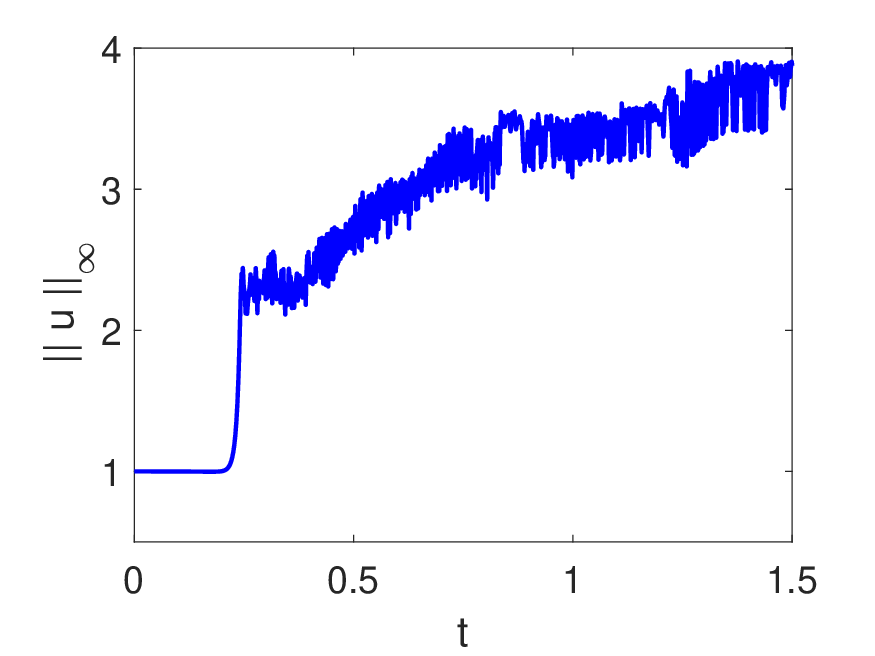}
 \end{minipage}
 \centering
 \begin{minipage}[t]{0.45\linewidth}
   \includegraphics[width=3in]{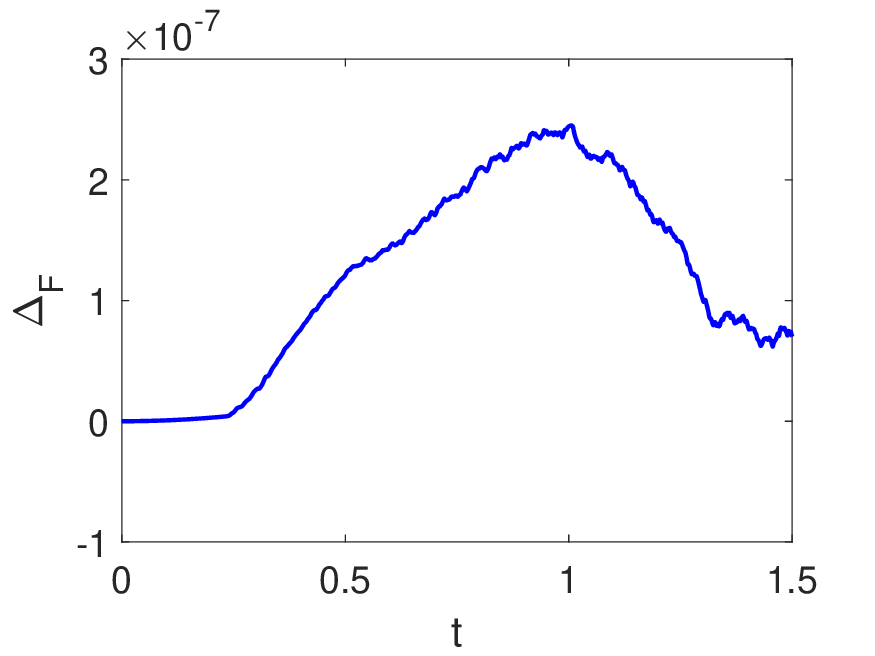}
 \end{minipage}
\caption{ The onset of breaking for the semi-classical nonlocal focusing  NLS equation  with real initial data for $\epsilon=0.1$ and $\beta=1.5$, the variation of $L^{\infty}$-norm of solution with time and the  variation of change in the
conserved quantity {${F}$} with time. } \label{shock1.5}
\end{figure}

\begin{figure}[h!]
 \begin{minipage}[t]{0.45\linewidth}
   \includegraphics[width=3in]{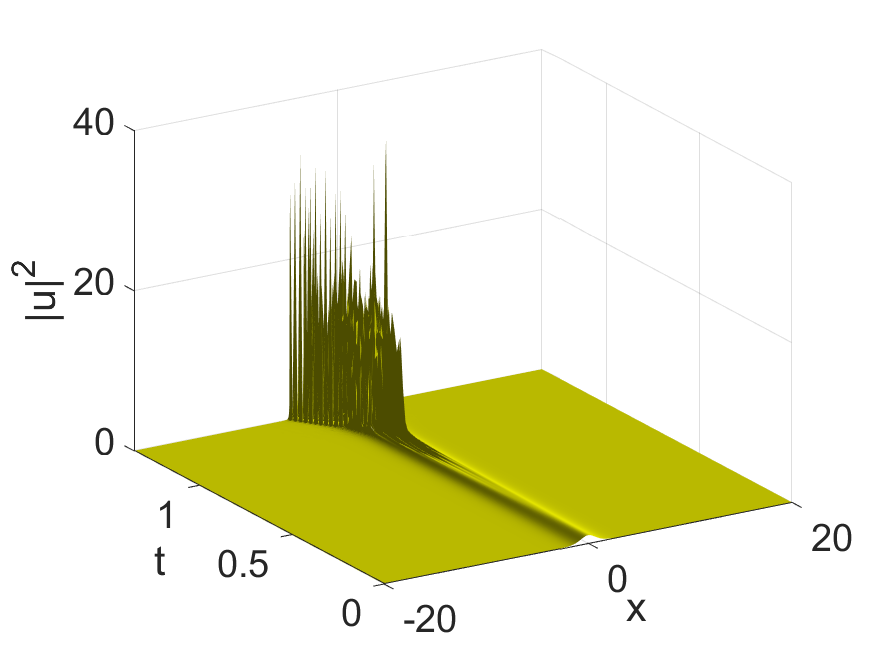}
 \end{minipage}
\hspace{30pt}
\begin{minipage}[t]{0.45\linewidth}
   \includegraphics[width=3in]{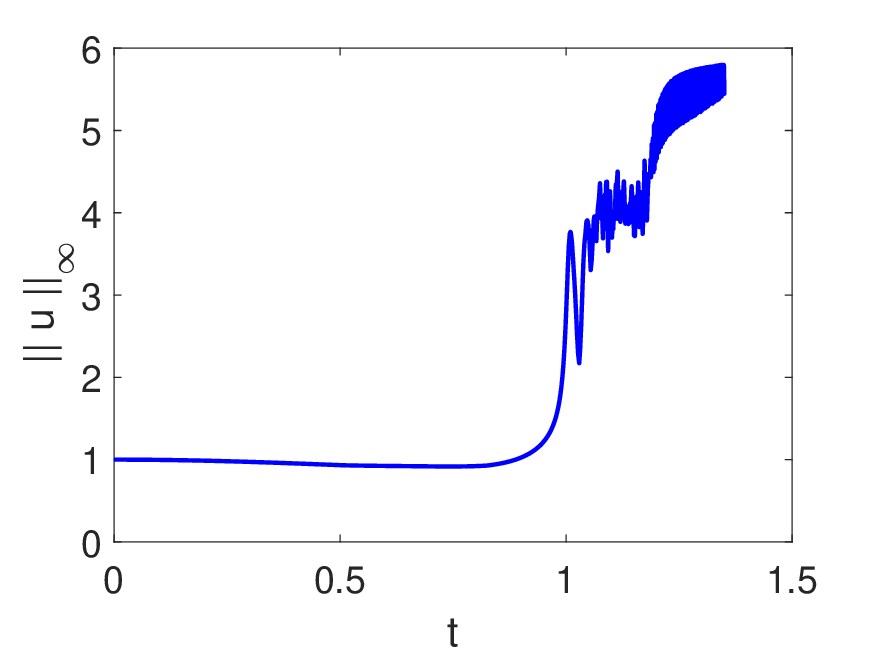}
 \end{minipage}
\caption{ The onset of breaking for the semi-classical nonlocal focusing  NLS equation  with real initial data for $\epsilon=0.1$ and $\beta=1$ and the variation of $L^{\infty}$-norm of solution with time.  } \label{shock1}
\end{figure}
\begin{figure}[h!]
 \begin{minipage}[t]{0.45\linewidth}
   \includegraphics[width=3in]{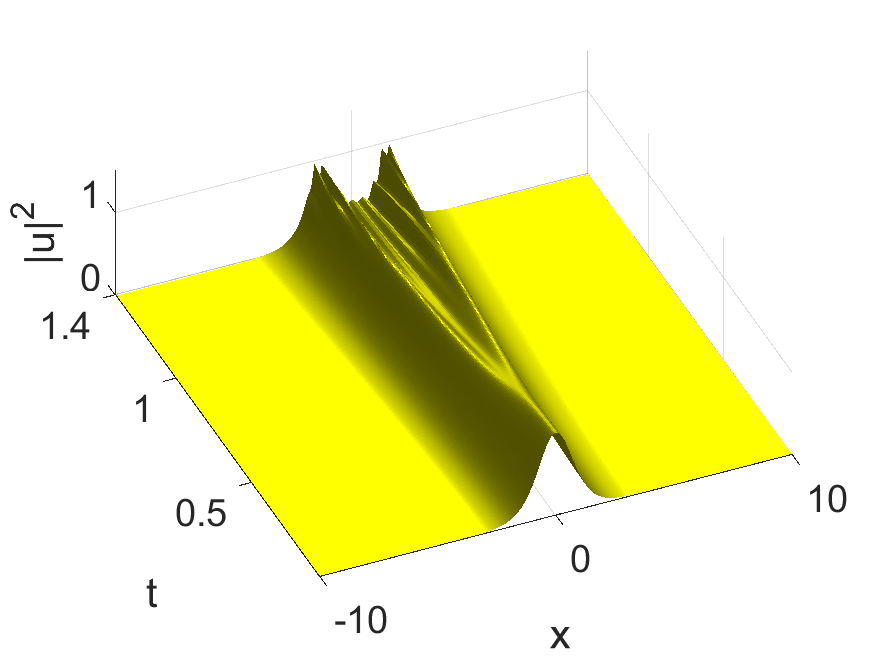}
 \end{minipage}
\hspace{30pt}
\begin{minipage}[t]{0.45\linewidth}
   \includegraphics[width=3in]{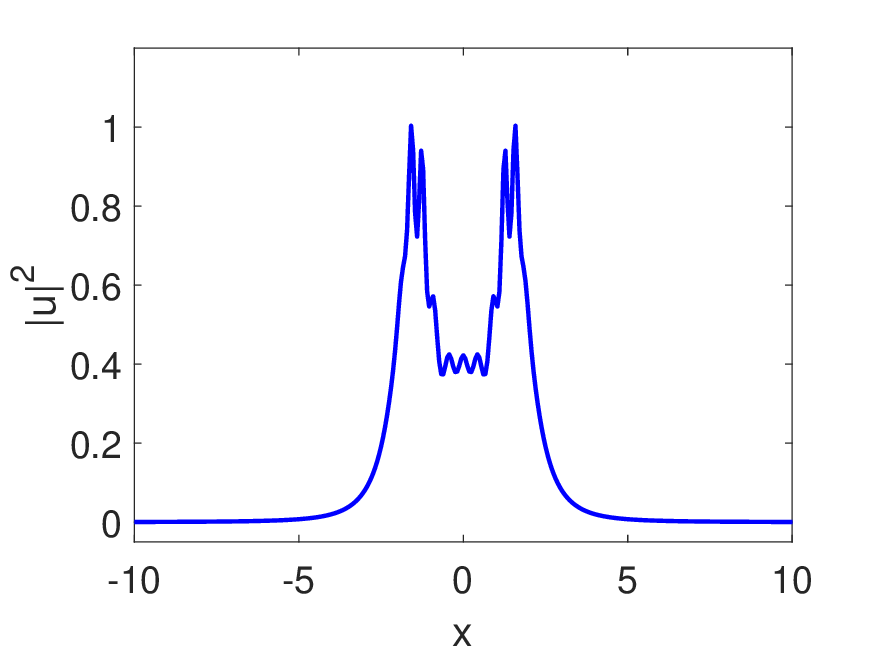}
 \end{minipage}
\caption{ The onset of breaking for the semi-classical nonlocal focusing NLS equation  with  the real initial condition for $\epsilon=0.1$, $\beta=0.7$ and the modulus squared of the solution at $t=1.4$. } \label{shock0.7focusing}
\end{figure}

\begin{figure}[h!]
 \begin{minipage}[t]{0.45\linewidth}
   \includegraphics[width=3in]{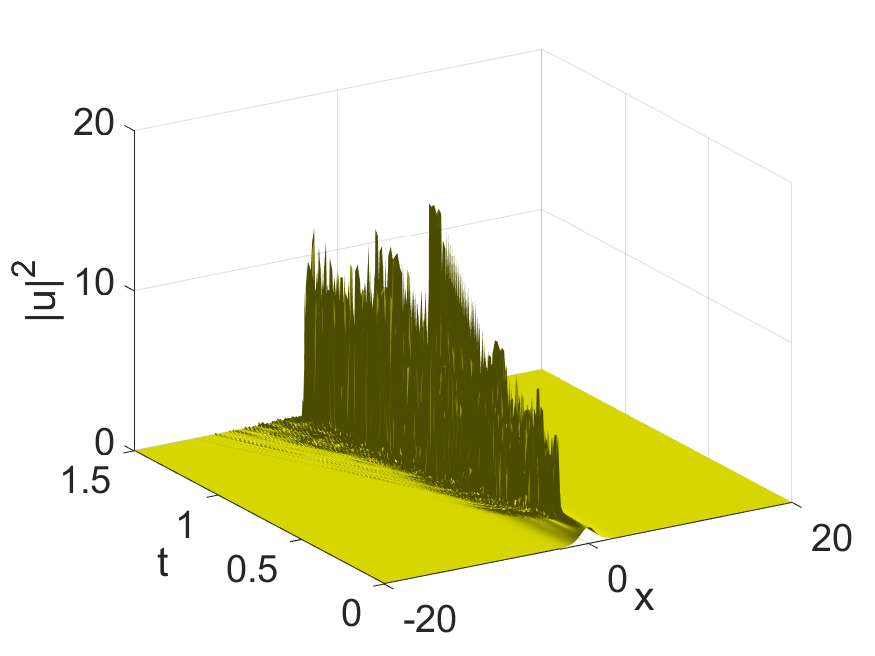}
 \end{minipage}
\hspace{30pt}
\begin{minipage}[t]{0.45\linewidth}
   \includegraphics[width=3in]{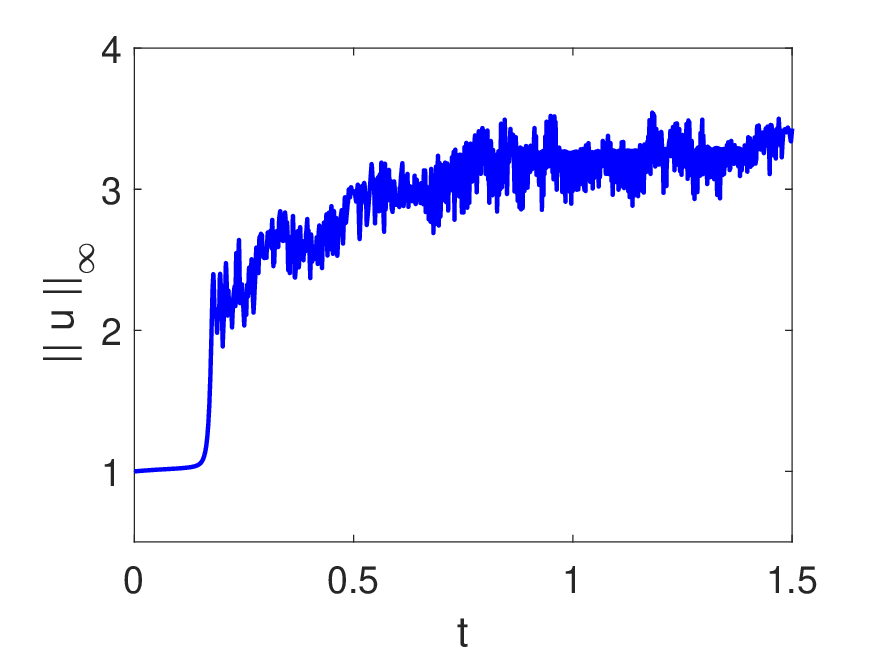}
 \end{minipage}
\caption{ The onset of breaking for the semi-classical nonlocal focusing NLS equation  with  the initial condition \eqref{phase} for $\epsilon=0.1$, $\beta=1.5$ and the variation of $L^{\infty}$-norm of solution with time.  } \label{shock1.5phase}
\end{figure}

\begin{figure}[h!]
 \begin{minipage}[t]{0.45\linewidth}
   \includegraphics[width=3in]{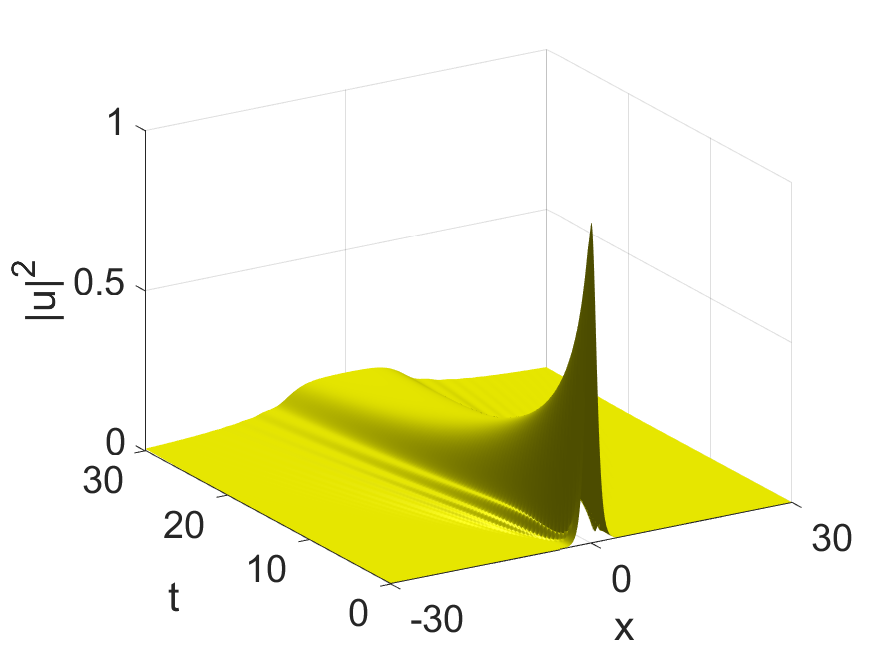}
 \end{minipage}
\hspace{30pt}
\begin{minipage}[t]{0.45\linewidth}
   \includegraphics[width=3in]{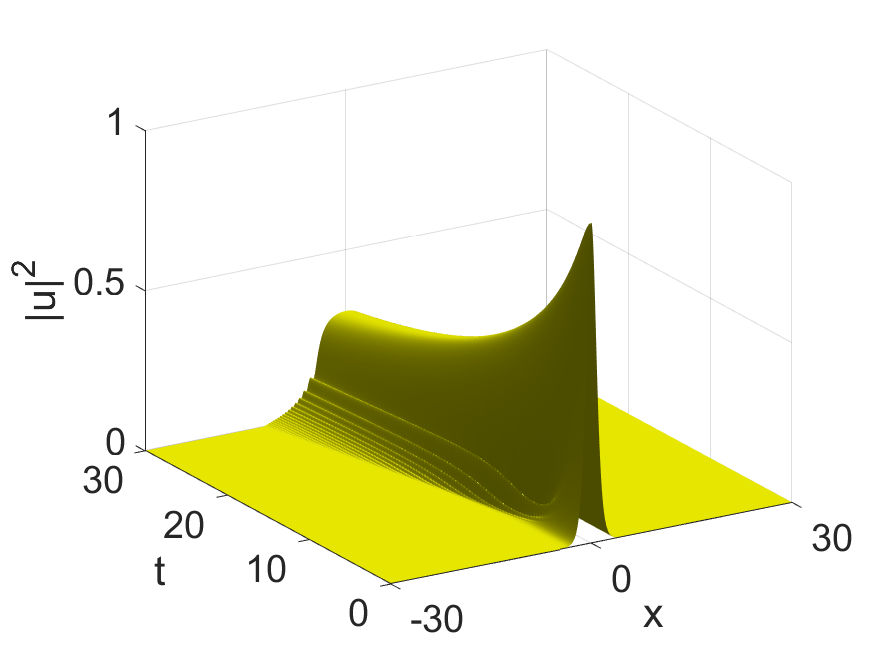}
 \end{minipage}
\caption{ Solution of the  semi-classical nonlocal defocusing NLS equation with  real initial condition  for  $\epsilon=0.1$ (left panel)  and $\epsilon=0.01$ (right panel).  } \label{shock1.5defocusing}
\end{figure}

\begin{figure}[h!]
 \begin{minipage}[t]{0.45\linewidth}
   \includegraphics[width=3in]{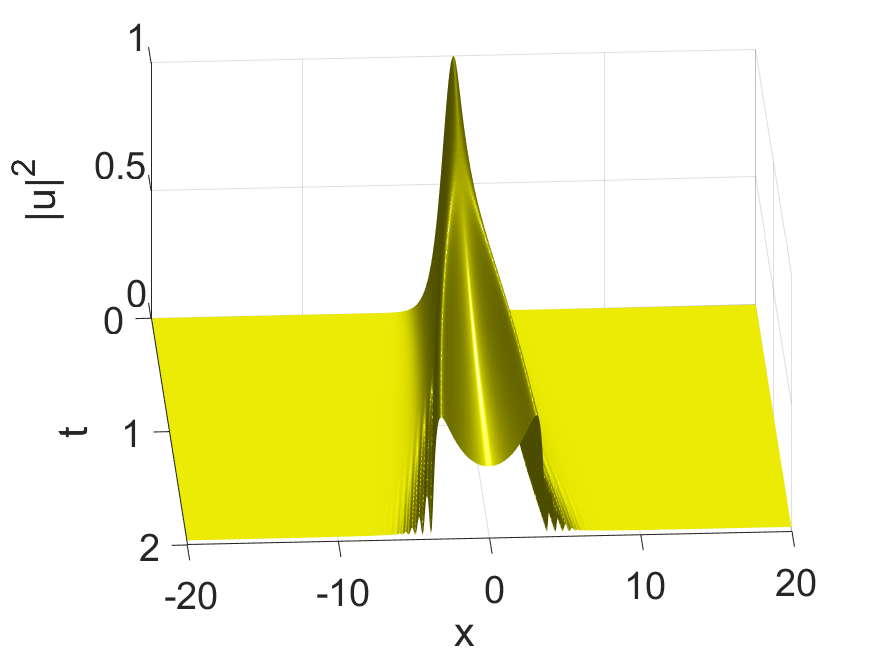}
 \end{minipage}
\hspace{30pt}
\begin{minipage}[t]{0.45\linewidth}
   \includegraphics[width=3in]{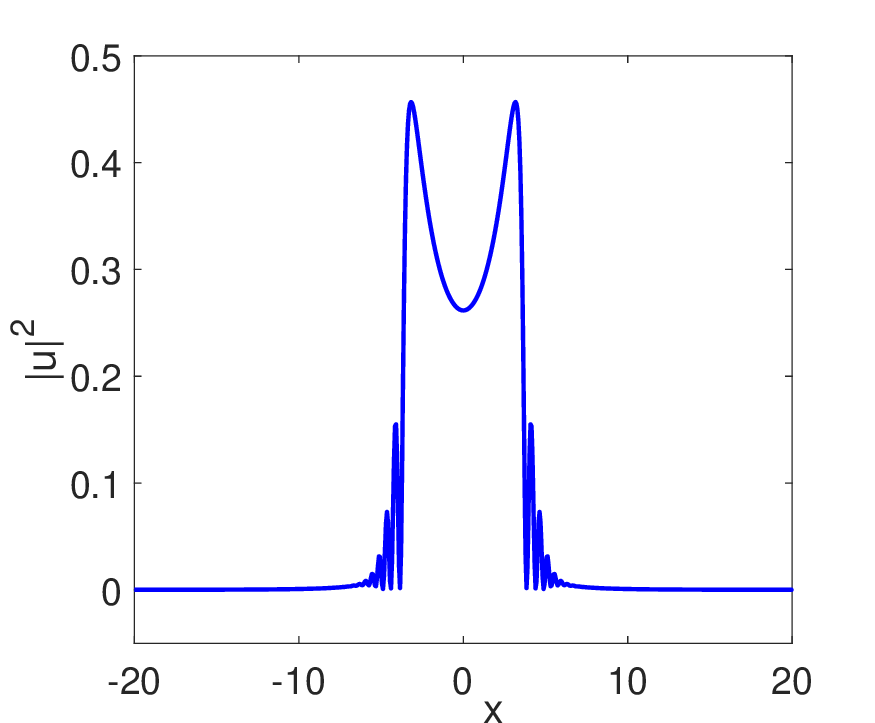}
 \end{minipage}
\caption{ Solution of the  semi-classical nonlocal defocusing NLS equation with  real initial condition  for  $\epsilon=0.1$, $\beta=0.7$  and the modulus squared of  the solution at $t=2$. } \label{shock0.7defocusing}
\end{figure}

This computational problem is notoriously
difficult. Numerical studies of the usual zero dispersion limits of nonlinear dispersive equations are hampered by the fact that in order to resolve the equation for a small parameter $\epsilon$, it requires a much tinier discretization.  The semi-classical initial condition is given by
\begin{equation}
u^{\epsilon}(x,0)=A(x) \exp \left(\ii \frac{S(x)}{\epsilon}\right),
\end{equation}
where $A(x)$ is the initial amplitude and
$S(x)$ is the real initial phase. We begin with the  nonlocal focusing NLS equation taking $\zeta=1$. Figure \ref{shock1.5} illustrates the self-focusing of a real initial condition
\begin{equation}
    u^{\epsilon}(x,0)=\mbox{sech}(x) \label{real}
\end{equation}
for $\epsilon=0.1$ and $\beta=1.5$.  The problem is solved in the space interval $-5000 \leq x \leq 5000$ up to $T=1.5$. We take the number of grid points as $N=2^{18}, M=30000$. For $\epsilon=0.1$ and $\beta=1.5$, the equation has weak nonlinearity.  As it is seen from the figure, self-focusing is followed by the onset of wave breaking and caustic formation similar to the semi-classical NLS equation in \cite{bronski}. The variation of $L^{\infty}$-norm of solution with time is also presented in Figure \ref{shock1.5}. We observe that the self-focusing occurs at the time of first break, around $t=0.2$. {Accuracy of the numerical scheme is checked by  the variation of change in the
conserved quantity ${ F}$ with time. Relative mass error stays around $10^{-7}$ during the computation.}
Figure \ref{shock1} depicts the  onset of breaking for the semi-classical nonlocal focusing NLS equation  with  \eqref{real} for $\epsilon=0.1$, $\beta=1$ and the variation of $L^{\infty}$-norm of solution with time.
The self-focusing occurs at the time of first break, around $t=0.98$. Figure \ref{shock0.7focusing} pictures the
evolution of real initial data for $\beta=0.7$ and the solution at $t=1.4$. Since the equation has strong nonlinearity, initial hump splits into two humps in a short time. Oscillatory region is observed.

Now we investigate  the effect of a non-trivial phase on a semiclassical evolution. The initial condition is chosen as
\begin{equation}
u^{\epsilon}(x,0)=\mbox{sech}(x) \exp \left(2\ii \frac{\mbox{sech}(x)}{\epsilon}\right) \label{phase}
\end{equation}
for $\epsilon=0.1$ and $\beta=1.5$. Figure \ref{shock1.5phase} shows the  onset of breaking for the semi-classical nonlocal NLS equation  with  \eqref{phase} for $\epsilon=0.1$, $\beta=1.5$ and the variation of $L^{\infty}$-norm of solution with time.
This figure is very similar to  the case of real initial data given in Figure \ref{shock1.5}. We observe that the self-focusing occurs at the time of first break, around $t=0.16$. The first break in Figure \ref{shock1.5phase}
appears a bit earlier  than  the first break in  Figure \ref{shock1.5}.

Next, we study the nonlocal defocusing NLS equation taking $\zeta=-1$.
Figure \ref{shock1.5defocusing} depicts the solution of the  semi-classical nonlocal defocusing NLS equation  corresponding to  real initial condition    for  $\beta=1.5, \epsilon=0.1$  and $\beta=1.5, \epsilon=0.01$.  Similar to observations for  defocusing fractional NLS \cite{klein} and  defocusing NLS  \cite{jin,bronski2}, the smooth initial pulse tends to `square up'. The initial hump  flattens while the sides of the hump steepen.
A train of rapid oscillations manifests  before and after the pulse. These oscillations become more visible
when $\epsilon$ decreases.  Figure \ref{shock0.7defocusing} illustrates the evolution of the  real initial condition for the   semi-classical nonlocal defocusing NLS equation  with    $\epsilon=0.1$  and $\beta=0.7$. In this case, the equation has strong nonlinearity. As it is seen from the figure, the initial hump splits into two humps. We observe some oscillations after the edge of the solution.


\subsection*{Conflict of interest} The authors declare that they have no conflict of interest.
	
	\subsection*{Data Availability}
	There is no data in this paper.

	\section*{Acknowledgment}
	 The authors   wish  to thank  four unknown referees and the editor for their
	 valuable suggestions which helped to improve the paper.
	
	A. E. is supported by Nazarbayev University under the Faculty Development Competitive Research Grants Program for 2023-2025 (grant number 20122022FD4121).
	

  \end{document}